\documentclass[graybox]{svmult}

\usepackage{mathptmx}       
\usepackage{helvet}         
\usepackage{courier}        
\usepackage{type1cm}        
\usepackage{makeidx}         
\usepackage{graphicx}        
\usepackage{multicol}        
\usepackage[bottom]{footmisc}
\usepackage{amsfonts}
\usepackage{amscd}
\usepackage{amssymb}
\usepackage{graphics}
\usepackage{amsmath}
\usepackage[english]{babel}
\usepackage{hyperref}

\hypersetup{colorlinks,citecolor=blue}

\input xy
\xyoption{all}

\makeindex             

\begin{document}

\title*{A Dual Interpretation of the Gromov--Thurston Proof of Mostow Rigidity and 
Volume Rigidity for Representations of Hyperbolic Lattices}
\titlerunning{Mostow Rigidity and Volume Rigidity for Hyperbolic Lattices}
\author{Michelle Bucher, Marc Burger and Alessandra Iozzi}
\authorrunning{M. Bucher, M. Burger and A. Iozzi} 
\institute{Michelle Bucher \at Section de Math\'ematiques Universit\'e de Gen\`eve, 
2-4 rue du Li\`evre, Case postale 64, 1211 Gen\`eve 4, Suisse, \email{Michelle.Bucher-Karlsson@unige.ch}
\and Marc Burger  \at Department Mathematik, ETH, R\"amistrasse 101, 
8092 Z\"urich, Schweiz, \email{burger@math.ethz.ch}
\and Alessandra Iozzi \at Department Mathematik, ETH, R\"amistrasse 101, 
8092 Z\"urich, Schweiz, \email{iozzi@math.ethz.ch}}

%
%
\maketitle

\abstract*{We use bounded cohomology to define a notion of volume of an 
$\mathrm{SO}(n,1)$-valued representation 
of a lattice $\Gamma<\mathrm{SO}(n,1)$ and, using this tool, 
we give a complete proof of the volume rigidity
theorem of Francaviglia and Klaff \cite{Francaviglia_Klaff} in this setting.  
Our approach gives in particular 
a proof of Thurston's version of Gromov's proof of Mostow Rigidity (also in the non-cocompact case), 
which is dual to the Gromov--Thurston proof using the simplicial volume invariant.}

\abstract{We use bounded cohomology to define a notion of volume of an 
$\mathrm{SO}(n,1)$-valued representation 
of a lattice $\Gamma<\mathrm{SO}(n,1)$ and, using this tool, 
we give a complete proof of the volume rigidity
theorem of Francaviglia and Klaff \cite{FranKlaff06} in this setting.  Our approach gives in particular 
a proof of Thurston's version of Gromov's proof of Mostow Rigidity (also in the non-cocompact case), 
which is dual to the Gromov--Thurston proof using the simplicial volume invariant.}

\section{Introduction} 
\label{sec:1}
Strong rigidity of lattices was proved in 1965 by Mostow \cite{Mostow} who,
while searching for a geometric explanation of the deformation rigidity results 
obtained by Selberg \cite{Selberg}, Calabi--Vesentini 
\cite{Calabi_Vesentini_1, Calabi_Vesentini_2} and Weil
\cite{Weil_1, Weil_2}, showed the remarkable fact 
that, under some conditions, topological data of a manifold determine its metric.  
Namely, he proved that if $M_i=\Gamma_i\backslash \mathbb{H}^n$, $i=1,2$ 
are compact quotients of real hyperbolic $n$-space and $n\geq3$, then  
any homotopy equivalence $\varphi:M_1\to M_2$ is, up  to homotopy, induced by an isometry.
Shortly thereafter, this was extended to the finite volume case by G.~Prasad \cite{Prasad}.

The methods introduced by Mostow emphasized the role of the quasi-isometries of 
$\widetilde M_i=\mathbb{H}^n$, 
their quasi-conformal extension to $\partial\mathbb{H}^n$, ergodicity phenomena of the $\Gamma_i$-action on
$\partial\mathbb{H}^n$, as well as almost everywhere differentiability results {\em \`a la} Egorov.  

In the 1970s, a new approach for rigidity in the real hyperbolic case was developed by Gromov. 
In this context he introduced $\ell^1$-homology and the simplicial volume: 
techniques like smearing and straightening became important. 
This approach was then further developed by Thurston \cite[Chapter 6]{Thurston_notes} 
and one of its consequences is an extension to hyperbolic manifolds of Kneser's theorem 
for surfaces \cite{Kneser}.   
To wit,  the computation of the simplicial volume $\|M\|=\frac{\mathrm{Vol}(M)}{v_n}$
implies for a continuous map  $f:M_1\to M_2$ between compact real hyperbolic manifolds, 
that 
$$
\deg f\leq\frac{\mathrm{Vol}(M_2)}{\mathrm{Vol}(M_1)}\,.
$$
If $\dim M_i\geq3$, Thurston proved that equality holds if and only if 
$f$ is homotopic to an isometric covering
while the topological assertion in the case in which $\dim M_i=2$ is Kneser's theorem \cite{Kneser}.

The next step, in the spirit of Goldman's theorem \cite{Goldman} -- 
what now goes under the theory of {\it maximal representations} -- 
is to associate an invariant $\mathrm{Vol}(\rho)$ to an arbitrary representation
$$
\rho:\pi_1(M)\to\mathrm{Isom}(\mathbb{H}^n)
$$ 
of the fundamental group of $M$, 
satisfying a Milnor--Wood type inequality
$$
\mathrm{Vol}(\rho)\leq\mathrm{Vol}(i)\,.
$$
The equality should be characterized as given by the ``unique" lattice embedding $i$ of $\pi_1(M)$, 
of course provided $\dim M\geq 3$.
This was carried out in $\dim M=3$ by Dunfield \cite{Dunfield}, following Toledo's modification
of the Gromov--Thurston approach to rigidity \cite{Toledo}.

If $M$ is only of finite volume, a technical difficulty is the definition of the volume 
$\mathrm{Vol}(\rho)$ of a representation.
Dunfield introduced for this purpose the notion of pseudodeveloping map and
Francaviglia proved that the definition is independent of the choice of the pseudodeveloping map 
\cite{Francaviglia}.
Then Francaviglia and Klaff \cite{FranKlaff06} proved a ``volume rigidity theorem" for representations
$$
\rho:\pi_1(M)\to\mathrm{Isom}(\mathbb{H}^k)\,,
$$
where now $k$ is not necessarily equal to $\dim M$.  
In their paper, the authors actually succeed in applying the technology developed 
by Besson--Courtois--Gallot in their seminal work on entropy rigidity \cite{Besson_Courtois_Gallot_gafa}.
An extension to representations of $\pi_1(M)$ into $\mathrm{Isom}(\mathbb{H}^n)$
for an arbitrary compact manifold $M$ has been given by Besson--Courtois--Gallot 
\cite{Besson_Courtois_Gallot_comm}.

Finally, Bader, Furman and Sauer proved a generalization of Mostow Rigidity for cocycles in the case of 
real hyperbolic lattices with some integrability condition, using, among others, bounded cohomology 
techniques, \cite{BaderFurmanSauer}.
\medskip

The aim of this paper is to give a complete proof of volume rigidity 
from the point of view of bounded cohomology,
implementing a strategy first described in \cite{Iozzi_ern} and 
used in the work on maximal representations
of surface groups \cite{Burger_Iozzi_Wienhard_ann, Burger_Iozzi_Wienhard_toledo}, 
as well as in the proof of
Mostow Rigidity in dimension $3$ in \cite{BurgerIozzi06}.

Our main contribution consists on the one hand in identifying the top dimensional bounded equivariant
cohomology of the full group of isometries $\mathrm{Isom}(\mathbb{H}^n)$, and on the other 
in giving a new definition of the volume of a representation of $\pi_1(M)$, when $M$ is not compact; 
this definition, that uses bounded relative cohomology, generalizes the one introduced in 
\cite{Burger_Iozzi_Wienhard_toledo} for surfaces.  

\medskip

In an attempt to be pedagogical, throughout the paper we try to describe, in varying details, 
the proof of all results.

\medskip
Let $\mathrm{Vol}_n(x_0,\dots,x_n)$ denote the signed volume of the convex hull of the points 
$x_0,\dots,x_n\in\overline{\mathbb{H}^n}$.
Then $\mathrm{Vol}_n$ is a $G^+:=\mathrm{Isom}^+(\mathbb{H}^n)$-invariant cocycle on 
$\overline{\mathbb{H}^n}$
and hence defines a top dimensional cohomology class $\omega_n\in H_c^n(G^+,\mathbf{R})$.
Let $i:\Gamma\hookrightarrow G^+$ be an embedding of $\Gamma$ as a lattice in the group 
of orientation preserving isometries of $\mathbb{H}^n$ and 
let $\rho:\Gamma\rightarrow G^+$ be an arbitrary representation of $\Gamma$. 
Suppose first that $\Gamma$ is torsion free. 
Recall that the cohomology of $\Gamma$ is canonically isomorphic 
to the cohomology of the $n$-dimensional quotient manifold $M:=i(\Gamma)\backslash \mathbb{H}^n$. 

If $M$ is compact, by Poincar\'e duality the cohomology groups 
$H^n(\Gamma,\mathbf{R})\cong H^n(M,\mathbf{R})$ 
in top dimension are canonically isomorphic to $\mathbf{R}$,
with the isomorphism  given by the evaluation on the fundamental class $[M]$. 
We define the volume $\mathrm{Vol}(\rho)$ of $\rho$ by 
$$
\mathrm{Vol}(\rho)=\langle \rho^*(\omega_n),[M]\rangle\,,
$$
where $\rho^*:H^n_c(G^+,\mathbf{R})\rightarrow H^n(\Gamma,\mathbf{R})$ 
denotes the pull-back via $\rho$. 
In particular  the absolute value of the volume of the lattice embedding $i$ 
is equal to the volume of the hyperbolic manifold $M$,
$\mathrm{Vol}(M)=\langle i^*(\omega_n),[M]\rangle$.

If $M$ is not compact, the above definition fails since $H^n(\Gamma,\mathbf{R})\cong H^n(M,\mathbf{R})=0$.
Thus we propose the following approach: since $\mathrm{Vol}_n$ is in fact a bounded cocycle,
it defines a bounded class $\omega_n^b\in H_{b,c}^n(G^+,\mathbf{R})$ in the bounded cohomology
of $G^+$ with trivial $\mathbf{R}$-coefficients.  Thus associated to a homomorphism $\rho:\Gamma\to G^+$
we obtain $\rho^\ast(\omega_n^b)\in H_b^n(\Gamma,\mathbf{R})$;
since $\widetilde M=\mathbb{H}^n$ is contractible, it follows easily that 
$H_b^n(\Gamma,\mathbf{R})$ is canonically isomorphic to the bounded singular cohomology
$H_b^n(M,\mathbf{R})$ of the manifold $M$ (this is true in much greater generality
\cite{Gromov_82, Brooks}, but it will not be used here).  To proceed further,
let $N\subset M$ be a compact core of $M$, that is the complement in $M$ of a disjoint union
of finitely many horocyclic neighborhoods $E_i$, $i=1,\dots, k$, of cusps.
Those have amenable fundamental groups and thus the map
$(N,\partial N)\to(M,\varnothing)$ induces an isomorphism in cohomology,
$H_b^n(N,\partial N,\mathbf{R})\cong H_b^n(M,\mathbf{R})$, by means of which 
we can consider $\rho^\ast(\omega_n^b)$ as a bounded relative class.
Finally, the image of $\rho^\ast(\omega_n^b)$ via the comparison map 
$c:H_b^n(N,\partial N,\mathbf{R})\to H^n(N,\partial N,\mathbf{R})$
is an ordinary relative class whose
evaluation on the relative fundamental class $[N,\partial N]$ gives the definition of the volume of $\rho$,
$$
\mathrm{Vol}(\rho):=\langle (c\circ\rho^\ast)(\omega_n^b),[N,\partial N]\rangle\,,
$$
which turns out to be independent of the choice of the compact core $N$.
When $M$ is compact, we recover of course the invariant previously defined.
We complete the definition in the case in which $\Gamma$ has torsion by setting 
\begin{equation*}
\mathrm{Vol}(\rho):=\frac{\mathrm{Vol}(\rho|_\Lambda)}{[\Gamma:\Lambda]}\,.
\end{equation*}
where $\Lambda<\Gamma$ is a torsion free subgroup of finite index.

\begin{theorem}\label{theorem:maxrep}
Let $n\geq3$. Let $i:\Gamma\hookrightarrow\mathrm{Isom}^+(\mathbb H^n)$ be a lattice
embedding and let $\rho:\Gamma\rightarrow \mathrm{Isom}^+(\mathbb H^n)$ be any representation. Then 
\begin{equation}\label{MilnorWood}
| \mathrm{Vol}(\rho)| \leq | \mathrm{Vol}(i)| =\mathrm{Vol}(M)\,,
\end{equation}
with equality if and only if $\rho$ is conjugated to $i$ by an isometry.
\end{theorem}

An analogous theorem, in the more general case of a representation 
$\rho:\Gamma\rightarrow \mathrm{Isom}^+(\mathbb H^m)$ with $m\geq n$,
has been proven by Francaviglia and Klaff \cite{FranKlaff06} 
with a different definition of volume. 

Taking in particular $\rho$ to be another lattice embedding of $\Gamma$, 
we recover Mostow--Prasad Rigidity theorem for hyperbolic lattices:

\begin{corollary}[\cite{Mostow, Prasad}] 
Let $\Gamma_{1},\Gamma_{2}$ be two isomorphic lattices in $\mathrm{Isom}^+(\mathbb{H}^{n})$. 
Then there exists an isometry $g\in\mathrm{Isom}(\mathbb{H}^{n})$ 
conjugating $\Gamma_1$ to $\Gamma_2$. 
\end{corollary}

As a consequence of Theorem~\ref{theorem:maxrep}, 
we also reprove Thurston's strict version of Gromov's degree inequality for hyperbolic manifolds. 
Note that this strict version generalizes Mostow Rigidity \cite[Theorem 6.4]{Thurston_notes}: 

\begin{corollary}[{\cite[Theorem 6.4]{Thurston_notes}}]\label{thm: Thurston Gromov Mostow}
Let $f:M_{1}\rightarrow M_{2}$ be a continuous proper map between two $n$-dimensional
hyperbolic manifolds $M_{1}$ and $M_{2}$ with $n\geq3$. Then 
$$
\mathrm{deg}(f)\leq\frac{\mathrm{Vol}(M_{2})}{\mathrm{Vol}(M_{1})}\,,
$$
with equality if and only if $f$ is homotopic to a local isometry. 
\end{corollary}

Our proof of Theorem~\ref{theorem:maxrep} follows closely the steps in the proof
of Mostow Rigidity. In particular, the following result  is the dual to the use
of measure homology and smearing in \cite{Thurston_notes}.
We denote by $\varepsilon:G\to\{-1,1\}$ the homomorphism defined by 
$\varepsilon(g)=1$ if $g$ is orientation preserving and $\varepsilon(g)=-1$ if
$g$ is orientation reversing.  

\begin{theorem}\label{thm:formula}  
Let $M=\Gamma\backslash \mathbb{H}^n$ be a finite volume real hyperbolic manifold. 
Let $\rho:\Gamma\to\mathrm{Isom}(\mathbb{H}^n)$ be a representation with non-elementary image and 
let $\varphi:\partial\mathbb{H}^n\to\partial \mathbb{H}^n$ be the corresponding equivariant measurable map.
Then for every $(n+1)$-tuple of points $\xi_0,\dots,\xi_n\in\partial\mathbb{H}^n$, 
\begin{equation}\label{formula}
\int_{\Gamma\backslash \mathrm{Isom}(\mathbb{H}^n)}
\varepsilon(\dot g^{-1})\mathrm{Vol}_n(\varphi(\dot g\xi_0),\dots,\varphi(\dot g\xi_n))\,d\mu(\dot g)=
\frac{\mathrm{Vol}(\rho)}{\mathrm{Vol}(M)}\mathrm{Vol}_n(\xi_0,\dots,\xi_n)\,,
\end{equation}
where $\mu$ is the invariant probability measure on $\Gamma\backslash \mathrm{Isom}(\mathbb{H}^n)$.
\end{theorem}

This allows us to deduce strong rigidity properties of the boundary map $\varphi$ 
from the cohomological information about the boundary that, in turn, are sufficient to show
the existence of an element $g\in\mathrm{Isom}^+(\mathbb H^n)$ conjugating $\rho$ and $i$.

To establish the theorem, we first prove the almost everywhere validity of the formula 
in Theorem~\ref{thm:formula}. 
Ideally, we would need to know that $H_{b,c}^n(G^+,\mathbf{R})$ is $1$-dimensional and 
has no coboundaries in degree $n$ in the appropriate cocomplex. 
However in general we do not know how to compute
$H_{b,c}^n(G^+,\mathbf{R})$, except when $G^+=\mathrm{Isom}^+(\mathbb{H}^2)$ 
or $\mathrm{Isom}^+(\mathbb{H}^3)$ and hence there is no direct way
to prove the formula in \eqref{formula}.  To circumvent this problem, we borrow
from \cite{Bucher_prod} (see also \cite{BucherMonod}) the essential 
observation that $\mathrm{Vol}_n$ is in fact a cocycle equivariant
with respect to the full group of isometries $G=\mathrm{Isom}(\mathbb{H}^n)$, that is
$$
\mathrm{Vol}_n(gx_1,\dots,gx_n)=\varepsilon(g)\mathrm{Vol}_n(x_1,\dots,x_n)\,.
$$
This leads to consider $\mathbf{R}$ as a non-trivial coefficient module
$\mathbf{R}_\varepsilon$ for $G$ and in this context we prove that the comparison map
$$
\xymatrix{H_{b,c}^n(G,\mathbf{R}_\varepsilon)\ar[r]^\cong
&H_c^n(G,\mathbf{R}_\varepsilon)
}
$$
is an isomorphism.
By a slight abuse of notation, 
we denote again by $\omega_n^b\in H_{b,c}^n(G,\mathbf{R}_\varepsilon)$
and by $\omega_n\in H_c^n(G,\mathbf{R}_\varepsilon)$ the generator defined by $\mathrm{Vol}_n$.

Using this identification and standard tools from the homological algebra approach to bounded cohomology,
we obtain the almost everywhere validity of the formula in Theorem~\ref{thm:formula}.
Additional arguments involving Lusin's theorem are required to establish the formula pointwise.
This is essential because one step of the proof (see the beginning of \S~\ref{sec:4})
consists in showing that, if there is the equality in \eqref{MilnorWood},
the map $\varphi$ maps the vertices of almost every positively oriented maximal ideal simplex 
to vertices of positively (or negatively - one or the other, not both) oriented maximal 
ideal simplices. Since such vertices form a set of measure zero in the boundary, 
an almost everywhere statement would not be sufficient. 

\section{The Continuous Bounded Cohomology of 
$G=\mathrm{Isom}(\mathbb{H}^{n})$}
\label{sec:2}

Denote by $G=\mathrm{Isom}(\mathbb{H}^{n})$ the full isometry group of hyperbolic $n$-space, 
and by $G^{+}=\mathrm{Isom}^{+}(\mathbb{H}^{n})$
its subgroup of index $2$ consisting of orientation preserving isometries.
As remarked in the introduction there are two natural $G$-module structures
on $\mathbf{R}$: the trivial one, which we denote by $\mathbf{R}$,
and the one given by multiplication with the homomorphism $\varepsilon:G\rightarrow G/G^{+}\cong\{+1,-1\}$,
which we denote by $\mathbf{R}_{\varepsilon}$. 

Recall that if $q\in\mathbb N$, 
the continuous cohomology groups $H_{c}^q(G,\mathbf{R})$, 
respectively $H_{c}^q(G,\mathbf{R}_\varepsilon)$ -- 
or in short $H_{c}^{\bullet}(G,\mathbf{R}_{(\varepsilon)})$ for both -- 
of $G$ with coefficient in $\mathbf{R}_{(\varepsilon)}$, 
is by definition given as the cohomology of the cocomplex
$$ 
\begin{aligned}
C_c(G^{q+1},\mathbf{R}_{(\varepsilon)})^G=
\{f:G^{q+1}\rightarrow \mathbf{R}_{(\varepsilon)} \mid &f\text{ is continuous and}\\
&\varepsilon(g)\cdot f(g_0,...,g_q)=f(gg_0,...,gg_q)\}
\end{aligned}
$$
endowed with its usual homogeneous coboundary operator
\begin{equation*}
\delta:C_c(G^{q+1},\mathbf{R}_{(\varepsilon)})^G\to C_c(G^{q+2},\mathbf{R}_{(\varepsilon)})^G
\end{equation*}
defined by
\begin{equation*}
\delta f(g_0,\dots,g_{q+1}):=\sum_{j=0}^{q+1}f(g_0,\dots,g_{j-1},g_{j+1},\dots,g_{q+1})\,.
\end{equation*}
This operator clearly restricts to the bounded cochains
$$
C_{c,b}(G^{q+1},\mathbf{R}_{(\varepsilon)})^G=\{ f\in C_c(G^{q+1},\mathbf{R}_{(\varepsilon)})^G 
\mid \| f\|_\infty=\sup_{g_0,...,g_q\in G}|f(g_0,...,g_q)|<+\infty \}
$$
and the continuous bounded cohomology $H_{c,b}^q(G,\mathbf{R}_{(\varepsilon)})$ of $G$ 
with coefficients in $\mathbf{R}_{(\varepsilon)}$ is the cohomology of this cocomplex. 
The inclusion
$$
C_{c,b}(G^{q+1},\mathbf{R}_{(\varepsilon)})^G\subset C_{c}(G^{q+1},\mathbf{R}_{(\varepsilon)})^G
$$ 
induces a comparison map
$$
c:H_{c,b}^q(G,\mathbf{R}_{(\varepsilon)})\longrightarrow H_{c}^q(G,\mathbf{R}_{(\varepsilon)})\,.
$$

We call cochains in $C_{c,(b)}(G^{q+1},\mathbf{R})^G$ invariant and 
cochains in $C_{c,(b)}(G^{q+1},\mathbf{R}_\varepsilon)^G$ equivariant and
apply this terminology to the cohomology classes as well. 
The sup norm on the complex of cochains induces a seminorm in cohomology
 
$$
\| \beta \|=\inf\{\| f\|_\infty \mid f\in C_{c,(b)}(G^{q+1},\mathbf{R}_{(\varepsilon)})^G, [f]=\beta\}\,,
$$
for $\beta \in H_{c,(b)}^q(G,\mathbf{R}_{(\varepsilon)})$. 

The same definition gives the continuous (bounded) cohomology of 
any topological group acting either trivially on $\mathbf{R}$ 
or via a homomorphism into the multiplicative group $\{+1,-1\}$. 
A continuous representation $\rho:H\rightarrow G$ naturally induces pullbacks 
\begin{equation*}
H_{c,(b)}^{\bullet}(G,\mathbf{R})\longrightarrow H_{c,(b)}^{\bullet}(H,\mathbf{R})\quad
\mathrm{and}\quad H_{c,(b)}^{\bullet}(G,\mathbf{R}_{\varepsilon})\longrightarrow H_{c,(b)}^{\bullet}(H,\mathbf{R}_\rho)\,,
\end{equation*}
where $\mathbf{R}_\rho$ is the $H$-module $\mathbb R$ with the $H$-action 
given by the composition of $\rho:H\rightarrow G$ with $\varepsilon:G\to\{+1,-1\}$. 
Note that  $\| \rho^*(\beta)\| \leq \| \beta \|.$

Since the restriction to $G^+$ of the $G$-action on $\mathbf{R}_{(\varepsilon)}$ is trivial, 
there is a restriction map in cohomomology 
\begin{equation}\label{equ:Coho Restriction to G+}
H_{c,(b)}^{\bullet}(G,\mathbf{R}_{(\varepsilon)})\longrightarrow H_{c,(b)}^{\bullet}(G^{+},\mathbf{R})\,.
\end{equation}
In fact, both the continuous and the continuous bounded cohomology groups 
can be computed isometrically on the hyperbolic $n$-space $\mathbb{H}^n$, 
as this space is isomorphic to the quotient of $G$ or $G^+$ by a maximal compact subgroup. 
More precisely, set 

$$
\begin{aligned}
C_{c,(b)}((\mathbb{H}^n)^{q+1},\mathbf{R}_{(\varepsilon)})^G
=\big\{f: (\mathbb{H}^n)^{q+1}\rightarrow \mathbf{R} \mid f \text{ is continuous (and bounded) and}&\\ 
 \varepsilon(g)\cdot f(x_0,...,x_q)=f(gx_0,...,gx_q)\big\}&
\end{aligned}
$$
and endow it with its homogeneous coboundary operator. 
Then the cohomology of this cocomplex is isometrically isomorphic 
to the corresponding cohomology groups (\cite[Ch. III, Prop. 2.3]{Guichardet} 
and \cite[Cor. 7.4.10]{Monod} respectively). 

It is now easy to describe the left inverses to the restriction map (\ref{equ:Coho Restriction to G+})
induced by the inclusion. Indeed, at the cochain level, they are given by maps 
$$
p: C_{c,(b)}((\mathbb{H}^n)^{q+1},\mathbf{R})^{G^+}\rightarrow C_{c,(b)}((\mathbb{H}^n)^{q+1},\mathbf{R})^G
$$ 
and 
$$
\overline{p}:C_{c,(b)}((\mathbb{H}^n)^{q+1},\mathbf{R})^{G^+}\rightarrow C_{c,(b)}((\mathbb{H}^n)^{q+1},\mathbf{R}_\varepsilon)^G
$$ 
defined for $x_0,...,x_q\in \mathbb{H}^n$ and $f\in C_{c,(b)}((\mathbb{H}^n)^{q+1},\mathbf{R})^{G^+}$ by 
\begin{eqnarray*}
p(f)(x_{0},...,x_{q}) & = & \frac{1}{2}\left(f(x_{0},...,x_{q})+f(\tau x_{0},...,\tau x_{q})\right),\\
\overline{p}(f)(x_{0},...,x_{q}) & = & \frac{1}{2}\left(f(x_{0},...,x_{q})-f(\tau x_{0},...,\tau x_{q})\right)\,,
\end{eqnarray*}
where $\tau\in G\smallsetminus G^{+}$ is any orientation reversing isometry.
In fact, it easily follows from the $G^+$-invariance of $f$ that $p(f)$
is invariant, $\overline{p}(f)$ is equivariant, and both $p(f)$
and $\overline{p}(f)$ are independent of $\tau$ in $G\smallsetminus G^{+}$. 
The following proposition is immediate:

\begin{proposition}\label{prop:decomposition}
The cochain map $(p,\overline{p})$ induces an isometric isomorphism
$$
H_{c,(b)}^{\bullet}(G^{+},\mathbf{R})\cong H_{c,(b)}^{\bullet}(G,\mathbf{R})\oplus H_{c,(b)}^{\bullet}(G,\mathbf{R}_{\varepsilon})\,. 
$$
\end{proposition}
The continuous cohomology group $H_{c}^{\bullet}(G^{+},\mathbf{R})$
is well understood since it can, via the van Est isomorphism \cite[Corollary~7.2]{Guichardet}, be identified
with the de Rham cohomology of the compact dual to $\mathbb{H}^{n}$,
which is the $n$-sphere $S^{n}$. Thus it is generated by two cohomology
classes: the constant class in degree $0$, and the volume form in
degree $n$. Recall that the volume form $\omega_n$ can be represented by the cocycle 
$\mathrm{Vol}_n\in C_{c,b}((\mathbb{H}^n)^{n+1},\mathbf{R}_\varepsilon)^G$ 
(respectively $\mathrm{Vol}_n\in L^\infty((\partial\mathbb{H}^n)^{n+1},\mathbf{R}_\varepsilon)^G$) given by
$$
\mathrm{Vol}_n(x_0,...,x_n)=\mathrm{signed \ volume \ of \ the \ convex \ hull \ of \ }x_0,...,x_n\,,
$$
for $x_0,...,x_n\in \mathbb{H}^n$, respectively $\partial \mathbb{H}^n$. 
Since the constant class in degree $0$ is invariant, and the volume form is equivariant, 
using Proposition~\ref{prop:decomposition} we summarize this as follows: 
$$
H_{c}^{0}(G^{+},\mathbf{R})\cong H_{c}^{0}(G,\mathbf{R})\cong\mathbf{R}
\quad\mathrm{and}\quad H_{c}^{n}(G^{+},\mathbf{R})\cong H_{c}^{n}(G,\mathbf{R}_{\varepsilon})
\cong\mathbf{R}\cong\left\langle \omega_n\right\rangle\,.
$$
All other continuous cohomology groups are $0$. 
On the bounded side, the cohomology groups are still widely unknown, 
though they are conjectured to be isomorphic to their unbounded counterparts. 
The comparison maps for $G$ and $G^+$ are easily seen to be isomorphisms in degree $2$ and $3$ 
(see \cite{BurgerIozzi06}).
We show that the comparison map for the {\it equivariant} cohomology of $G$ is 
indeed an isometric isomorphism up to degree $n$, 
based on the simple Lemma~\ref{lemma: no coboundary} below. 
Before we prove it, it will be convenient to have yet two more cochain complexes
to compute the continuous bounded cohomology groups. 
If $X=\mathbb{H}^n$ or $X=\partial\mathbb{H}^n$, 
consider the cochain space $L^{\infty}(X^{q+1},\mathbf{R}_{(\varepsilon)})^{G} $ 
of $G$-invariant, resp. $G$-equivariant, 
essentially bounded measurable function classes endowed with its homogeneous coboundary operator. 
It is proven in \cite[Cor. 7.5.9]{Monod} that the cohomology of this cocomplex 
is isometrically isomorphic to the continuous bounded cohomology groups.
Note that the volume cocycle $\mathrm{Vol}_n$ represents the same cohomology class viewed
as continuous bounded or $L^\infty$-cocycle on $\mathbb H^n$,
as an $L^\infty$-cocyle on $\partial\mathbb H^n$ or,
by evaluation on $x\in\mathbb H^n$ or $x\in\partial\mathbb H^n$,
as a continuous bounded or $L^\infty$-cocycle on $G$.

\begin{lemma}\label{lemma: no coboundary}
For $q<n$ we have
\begin{eqnarray*}
C_{c}((\mathbb{H}^{n})^{q+1},\mathbf{R}_{\varepsilon})^{G} & = & 0\,,\\
L^{\infty}((\mathbb{H}^{n})^{q+1},\mathbf{R}_{\varepsilon})^{G} & = & 0\, , \\
L^{\infty}((\partial\mathbb{H}^{n})^{q+1},\mathbf{R}_{\varepsilon})^{G} & = & 0\,.
\end{eqnarray*}
\end{lemma}
\begin{proof}
Let $f:(\mathbb{H}^{n})^{q+1}\rightarrow\mathbf{R}_\varepsilon$ or 
$f:(\partial\mathbb{H}^{n})^{q+1}\rightarrow\mathbf{R}_\varepsilon$
be $G$-equivariant. The lemma relies on the simple observation that
any $q+1\leq n$ points $x_{0},...,x_{q}$  either in $\mathbb{H}^{n}$
or in $\partial\mathbb{H}^{n}$ lie either on a hyperplane $P\subset\mathbb{H}^{n}$
or on the boundary of a hyperplane. Thus there exists a reversing
orientation isometry $\tau\in G\smallsetminus G^{+}$ fixing $(x_0,\dots,x_q)$ pointwise.
Using the $G$-equivariance of $f$ we conclude that
$$
f(x_{0},...,x_{q})=-f(\tau x_{0},...,\tau x_{q})=-f(x_{0},...,x_{q})\,,
$$
which implies $f\equiv0$.\qed
\end{proof}
It  follows from the lemma that 
$H_{c,b}^{q}(G,\mathbf{R}_{\varepsilon})\cong H_{c}^{q}(G,\mathbf{R}_{\varepsilon})=0$
for $q<n$. Furthermore, we can conclude that the comparison map for
the equivariant cohomology of $G$ is injective: 

\begin{proposition}\label{prop: compmapisomdegn}
The comparison map induces an isometric isomorphism 
$$
\xymatrix@1
{H_{c,b}^{n}(G,\mathbf{R}_{\varepsilon})\ar[r]^\cong 
&H_{c}^{n}(G,\mathbf{R}_{\varepsilon})\,.
}
$$
\end{proposition}
\begin{proof}
Since there are no cochains in degree $n-1$, there are no coboundaries
in degree $n$ and the cohomology groups $H_{c,b}^{n}(G,\mathbf{R}_{\varepsilon})$
and $H_{c}^{n}(G,\mathbf{R}_{\varepsilon})$ are equal to the corresponding spaces of cocycles. 
Thus, we have a commutative diagram
\begin{equation*}
\xymatrix{ H^n_{c,b}(G,\mathbf{R}_\varepsilon) \ar@{=}[r] \ar[d] & \mathrm{Ker}
\{ \delta: C_{c,b}(( \mathbb{H}^n)^{n+1} ,\mathbf{R}_\varepsilon)^G\rightarrow 
C_{c,b}(( \mathbb{H}^n)^{n+2},\mathbf{R}_\varepsilon)^G \} \ar@{^{(}->}[d]\\
\mathbf{R}\cong H^n_{c}(G,\mathbf{R}_\varepsilon) \ar@{=}[r] & \mathrm{Ker}
\{ \delta:C_{c}((\mathbb{H}^n)^{n+1} ,\mathbf{R}_\varepsilon)^G\rightarrow 
C_{c}(( \mathbb{H}^n)^{n+2},\mathbf{R}_\varepsilon)^G \}},
\end{equation*}
and the proposition follows from the fact that 
the lower right kernel is generated by the volume form $\omega_n$ 
which is represented by the bounded cocycle $\mathrm{Vol}_n$ which is in the image of the vertical right inclusion.
\qed
\end{proof}

Since there are no coboundaries in degree $n$ in $C_{c}((\mathbb{H}^{n})^{q+1},\mathbf{R}_{\varepsilon})^{G} $, 
it follows that the cohomology norm of $\omega_n$ is equal 
to the norm of the unique cocycle representing it. 
In view of \cite{HaagMunk}, 
its norm is equal to the volume $v_n$ of an ideal regular simplex in $\mathbb{H}^n$.

\begin{corollary}
The norm $\left\Vert \omega_{n}\right\Vert $
of the volume form $\omega_{n}\in H_{c}^{n}(G^{+},\mathbf{R})$
is equal to the volume $v_n$ of a regular ideal simplex in $\mathbb{H}^{n}$.
\end{corollary}

As the cohomology norm $\| \omega_n\|$ is the proportionality constant 
between simplicial and Riemannian volume for closed hyperbolic manifolds \cite[Theorem 2]{BucherProp}, 
the corollary gives a simple proof of the proportionality principle 
$\| M\| =\mathrm{Vol}(M)/v_n$ for closed hyperbolic manifolds, 
originally due to Gromov and Thurston.

\section{Relative Cohomology}

\subsection{Notation and Definitions}

As mentioned in the introduction, 
we consider a compact core $N$ of the complete hyperbolic manifold $M$,
that is a subset of $M$ whose complement $M\smallsetminus N$ in $M$
is a disjoint union of finitely many geodesically convex cusps of $M$. 
If $q\geq0$ and $\sigma:\Delta^q\to M$ denotes a singular simplex, 
where 
$\Delta^q=\{(t_0,\dots, t_q)\in\mathbf{R}^{q+1}:\,\sum_{j=0}^qt_j=1, \,t_j\geq0\text{ for all }j\}$ 
is a standard $q$-simplex,
we recall that the (singular) cohomology $H^q(M,M\smallsetminus N)$ of $M$ 
relative to $M\smallsetminus N$ is the cohomology of the cocomplex
$$
C^q(M,M\smallsetminus N)=\{f\in C^q(M)\mid f(\sigma)=0 
\mathrm{ \ if \ } \mathrm{Im}(\sigma)\subset M\smallsetminus N\}
$$
endowed with its usual coboundary operator. 
(Here, $C^q(M)$ denotes the space of singular $q$-cochains on $M$.) 
We emphasize that all cohomology groups, singular or relative, are with $\mathbf{R}$ coefficients.
The bounded relative cochains $C^q_b(M,M\smallsetminus N)$ are 
those for which $f$ is further assumed to be bounded, 
meaning that $\mathrm{sup}\{|f(\sigma)|\mid \sigma:\Delta^q\rightarrow M\}$ is finite. 
The coboundary restricts to bounded cochains and the cohomology of that cocomplex 
is the bounded cohomology of $M$ relative to $M\smallsetminus N$, 
which we denote by $H^\bullet_b(M,M\smallsetminus N)$. 
The inclusion of cocomplexes induces a comparison map 
$c: H^\bullet_b(M,M\smallsetminus N)\rightarrow H^\bullet(M,M\smallsetminus N)$. 
Similarly, we could define the cohomology of $N$ relative to its boundary $\partial N$ and 
it is clear, by homotopy invariance, 
that $H^\bullet_{(b)}(N,\partial N)\cong H^\bullet_{(b)}(M,M\smallsetminus N)$. 
We can identify the relative cochains on $(M,M\smallsetminus N)$ 
with the $\Gamma$-invariant relative cochains $C^q(\mathbb{H}^n,U)^\Gamma$ 
on the universal cover $\mathbb{H}^n$ relative to the preimage $U=\pi^{-1}(M\smallsetminus N)$ 
under the covering map $\pi:\mathbb{H}^n\rightarrow M$ 
of the finite union of horocyclic neighborhoods of cusps. 
We will identify $H^\bullet_{(b)}(N,\partial N)$ with the latter cohomology group. 
Note that  $U$ is a countable union of disjoint horoballs. 

The inclusion $(M,\varnothing)\hookrightarrow (M,M\smallsetminus N) $ 
induces a long exact sequence on both the unbounded and bounded cohomology groups
$$\dots\longrightarrow  H^{\bullet-1}_{(b)}(M\smallsetminus N)\longrightarrow 
H^{\bullet}_{(b)}(M,M\smallsetminus N)\longrightarrow 
H^{\bullet}_{(b)}(M)\longrightarrow 
H^{\bullet}_{(b)}(M\smallsetminus N)\longrightarrow 
\dots
$$
Each connected component $E_j$ of $M\smallsetminus N$, $1\leq j\leq k$, is a horocyclic 
neighborhood of a cusp, hence homeomorphic to the product 
of $\mathbf R$ with a torus; thus its universal covering is contractible and its
fundamental group is abelian (hence amenable).  
It follows that (see the introduction or \cite{Gromov_82, Brooks})
$H^\bullet_b(E_j)\cong H^\bullet_b(\pi_1(E_j))=0$ and hence 
$H^\bullet_b(M\smallsetminus N)=0$, 
proving that the inclusion $(M,\varnothing)\hookrightarrow (M,M\smallsetminus N) $ 
induces an isomorphism on the bounded cohomology groups. 
Note that based on some techniques developed in \cite{BucherKimKim} 
we can show that this isomorphism is isometric - 
a fact that we will not need in this note. 

\subsection{Transfer Maps}
In the following we identify $\Gamma$ with its image $i(\Gamma)<G^+$ 
under the lattice embedding $i:\Gamma\to G^+$.
There exist natural transfer maps 
\begin{equation*}
\xymatrix{ H^\bullet_b(\Gamma) \ar[r]^{\operatorname{trans}_\Gamma \ \ }
& H^\bullet_{c,b}(G,\mathbf{R}_\varepsilon) 
& \mathrm{and} 
& H^\bullet(N,\partial N)\ar[r]^{\tau_{dR}\ \  }
& H^\bullet_{c}(G,\mathbf{R}_\varepsilon),  }
\end{equation*}
whose classical constructions we briefly recall here. 
The aim of this section will then be to establish the commutativity of the diagram 
(\ref{equ: CD with trans}) in Proposition~\ref{prop: comm}. 
The proof is similar to that in \cite{BucherKimKim}, 
except that we replace the compact support cohomology by the relative cohomology, 
which leads to some simplifications. 
In fact, the same proof as in \cite{Burger_Iozzi_Wienhard_toledo}
(from where the use of relative bounded cohomology is borrowed)
would have worked {\em verbatim} in this case, 
but we chose the other (and simpler), 
to provide a ``measure homology-free'' proof.  

\subsubsection*{The Transfer Map 
$\operatorname{trans}_\Gamma:H^\bullet_b(\Gamma)\rightarrow H^\bullet_{c,b}(G,\mathbf{R}_\varepsilon)$}
We can define the transfer map at the cochain level either as a map 
$$
\operatorname{trans}_\Gamma:V_q^\Gamma\to V_{q}^G\,,
$$
where $V_q$ is one of $C_{b}((\mathbb{H}^n)^{q+1},\mathbf{R})$,
$L^\infty((\mathbb{H}^n)^{q+1},\mathbf{R})$
or $L^\infty((\partial\mathbb{H}^n)^{q+1},\mathbf{R})$.
The definition is the same in all cases. 
Let thus $c$ be a $\Gamma$-invariant cochain in $V_q^\Gamma$. Set
\begin{equation}\label{trans}
\operatorname{trans}_\Gamma(c)(x_0,...,x_n):=
\int_{\Gamma\backslash G} \varepsilon (\dot g^{-1})\cdot c(\dot gx_0,...,\dot gx_n)d\mu(\dot g)\,,
\end{equation}
where $\mu$ is the invariant probability measure on $\Gamma\backslash G$ normalized so that 
$\mu(\Gamma\backslash G)=1$. Recall that $\Gamma<G^+$, so that $\varepsilon(\dot g)$ is well defined.
It is easy to check that 
the resulting cochain $\operatorname{trans}_{\Gamma}(c)$ is $G$-equivariant. 
Furthermore, the transfer map clearly commutes with the coboundary operator, 
and hence induces a cohomology map
$$
\xymatrix{H^\bullet_b(\Gamma) \ar[r]^-{\operatorname{trans}_\Gamma \ \ }& H^\bullet_{c,b}(G,\mathbf{R}_\varepsilon)\,.}
$$
Note that if the cochain $c$ is already $G$-equivariant, 
then $\operatorname{trans}_\Gamma(c)=c$, showing that $\operatorname{trans}_\Gamma$ 
is a left inverse of $i^*:H^\bullet_{c,b}(G,\mathbf{R}_\varepsilon)\rightarrow H^\bullet_b(\Gamma)$.

\subsubsection*{The Transfer Map  
$\tau_{dR}:H^\bullet(N,\partial N)\rightarrow H^\bullet_{c}(G,\mathbf{R}_\varepsilon)$}

The relative de Rham cohomology  $H^\bullet_{dR}(M,M\smallsetminus N)$ 
is the cohomology of the cocomplex of differential forms $\Omega^q(M,M\smallsetminus N)$ 
which vanish when restricted to $M\smallsetminus N$. 
Then, as for usual cohomology, there is a de Rham Theorem 
$$
\xymatrix@1{
\Psi:H^\bullet_{dR}(M,M\smallsetminus N )\ar[r]^-\cong 
&H^\bullet(M,M\smallsetminus N)\cong H^\bullet(N,\partial N)}
$$ for relative cohomology. 
The isomorphism is given at the cochain level by integration. In order to integrate, 
we could either replace the singular cohomology by its smooth variant 
(i.e. take smooth singular simplices), 
or we prefer here to integrate the differential form on the straightened simplices. 
(The geodesic straightening of a continuous simplex is always smooth.) 
Thus, at the cochain level, the isomorphism is induced by the map 
\begin{equation}\label{psi}
\Psi: \Omega^q(M,M\smallsetminus N) \longrightarrow C^q(M,M\smallsetminus N)\,,
\end{equation}
sending a differential form 
$\omega\in \Omega^q(M,M\smallsetminus N)\cong \Omega^q(\mathbb{H}^n,U)^\Gamma$ 
to the singular cochain $\Psi(\omega)$ given by 
$$
\sigma \mapsto \int_{\pi_* \mathrm{straight}(x_0,...,x_q)} \omega\,,
$$
where $\pi:\mathbb{H}^n\rightarrow M$ is the canonical projection, 
the $x_i\in \mathbb{H}^n$ are the vertices of a lift of $\sigma$ to $\mathbb{H}^n$, and 
$\mathrm{straight}(x_0,...,x_q):\Delta^q\rightarrow \mathbb{H}^n$ is the geodesic straightening. 
Observe that if $\sigma$ is in $U$, then the straightened simplex is as well, 
since all components of $U$ are geodesically convex.

The transfer map 
$\operatorname{trans}_{dR}:H^\bullet_{dR}(M,M\smallsetminus N)\rightarrow H^\bullet_c(G,\mathbb{R_\varepsilon})$ 
is defined through the relative de Rham cohomology 
and the van Est isomorphism. 
At the cochain level the transfer
$$
\operatorname{trans}_{dR}:\Omega^q(\mathbb{H}^n,U)^\Gamma \longrightarrow \Omega^q(\mathbb{H}^n,\mathbf{R}_\varepsilon)^G
$$
is defined by sending the differential $q$-form $\alpha\in \Omega^q(\mathbb{H}^n)^\Gamma$ to the form 
$$
\operatorname{trans}_{dR}(\alpha):=\int_{\Gamma\backslash G} \varepsilon(\dot g^{-1})\cdot( \dot g^*\alpha) d\mu(\dot g)\,,
$$
where $\mu$ is chosen as in \eqref{trans}. 
It is easy to check that the resulting differential form 
$\operatorname{trans}_{dR}(\alpha)$ is $G$-equivariant. 
Furthermore, the transfer map clearly commutes with the differential operator, and 
hence induces a cohomology map
$$
\xymatrix{
H^\bullet(N,\partial N)
& 
&H^\bullet_c(G,\mathbf{R}_\varepsilon)\\
H^\bullet_{dR}(M,M\smallsetminus N)\ar[u]_{\cong}^{\Psi}\ar[r]^-{\operatorname{trans}_{dR}}
&H^\bullet(\Omega^\bullet(\mathbb{H}^n,\mathbf{R}_\varepsilon)^G)\ar[r]^-= 
&\Omega^\bullet(\mathbb{H}^n,\mathbf{R}_\varepsilon)^G\,,\ar[u]_-\cong
}
$$
where the vertical arrow on the right is the van Est isomorphism
and the horizontal arrow on the right follows
from Cartan's lemma to the extent that any $G$-invariant differential form
on $\mathbb{H}^n$ (or more generally on a symmetric space) is closed.

Let $\omega_{N,\partial N}\in H^n(M,M\smallsetminus N)$ be the unique class with 
$\langle \omega_{N,\partial N},[N,\partial N]\rangle=\mathrm{Vol}(M)$. 
It is easy to check that 
\begin{equation}\label{volume of M}
\operatorname{trans}_{dR}(\omega_{N,\partial N})=\omega_n\in\Omega^n(\mathbb{H}^n,\mathbb{R}_\varepsilon)^G
\cong H^n_c(G,\mathbf{R}_\varepsilon)\,.
\end{equation}

\subsubsection*{Commutativity of the Transfer Maps}

\begin{proposition}\label{prop: comm} The diagram 
\begin{equation}\label{equ: CD with trans}
\xymatrix{ H^q_b(\Gamma) \ar[rd]^-{\operatorname{trans}_\Gamma} &\\
H^q_{b}(N,\partial N) \ar[u]^-{\cong}\ar[d]^-{c}& H^q_{c,b}(G,\mathbf{R}_\varepsilon)\ar[d]_-{c} \\
H^q(N,\partial N)\ar[r]^-{\tau_{dR}}& H^q_{c}(G,\mathbf{R}_\varepsilon).}
\end{equation}
commutes, where $\tau_{dR}=\operatorname{trans}_{dR}\circ\Psi^{-1}$. 
\end{proposition}

\begin{proof} The idea of the proof is to subdivide the diagram (\ref{equ: CD with trans}) 
in smaller parts, 
by defining transfer maps directly on the bounded and unbounded relative
singular cohomology of $M$ 
and show that each of the following subdiagrams commute. 

\begin{equation}\label{equ: CD with all trans}
\xymatrix{ H^q_b(\Gamma) \ar[rd]^{\operatorname{trans}_\Gamma}&\\
H^q_{b}(N,\partial N)  \ar[u]^{\cong}\ar@{-->}[r]^{\operatorname{trans}_b}\ar[d]^{c}& H^q_{c,b}(G,\mathbf{R}_\varepsilon)\ar[d]_{c} \\
H^q(N,\partial N)\ar@{-->}[r]^{\operatorname{trans}}& H^q_{c}(G,\mathbf{R}_\varepsilon) \\
H^q_{dR}(N,\partial N)\ar[r]^{\operatorname{trans}_{dR}} \ar[u]^\cong_\Psi& \Omega^q(\mathbb{H}^n,\mathbf{R}_\varepsilon)^G.\ar[u]_\cong^\Phi}
\end{equation}

\subsubsection*{Definition of the Transfer Map for Relative Cohomology} 
In order to define a transfer map, 
we need to be able to integrate our cochain on translates of a singular simplex 
by elements of $\Gamma\backslash G$. This is only possible if the cochain is regular enough. 

For $1\leq i \leq k$, pick a point $b_i\in E_i$ in each horocyclic neighborhood of a cusp in $M$ and $b_0\in N$ in the compact core. 
Let $\beta':M\rightarrow \{b_0,b_1,...,b_k\}$ be the measurable map 
sending $N$ to $b_0$ and each cusp $E_i$ to $b_i$. 
Lift $\beta'$ to a $\Gamma$-equivariant measurable map
$$
\beta:\mathbb{H}^n\longrightarrow\pi^{-1}(\{b_0,b_1,...,b_k\})\subset \mathbb{H}^n
$$
defined as follows.  Choose lifts $\tilde b_0,\dots, \tilde b_k$ of $b_0,\dots, b_k$;
for each $j=1,\dots, k$ choose a Borel fundamental domain $\mathcal D_j\ni\tilde b_j$
for the $\Gamma$-action on $\pi^{-1}(E_j)$ and choose a fundamental domain
$\mathcal D_0\ni\tilde b_0$ for the $\Gamma$-action on $\pi^{-1}(N)$.  Now define
$\beta(\gamma\mathcal D_j):=\gamma\tilde b_j$.  
In particular $\beta$  maps each horoball into itself. 
Given $c\in C^q(\mathbb{H}^n,U)^\Gamma$, define
$$\beta^*(c):(\mathbb{H}^n)^{q+1}\longrightarrow \mathbf{R}$$
by
\begin{equation}\label{beta-ast}
\beta^*(c)(x_0,...,x_q)=c(\mathrm{straight}(\beta(x_0),...,\beta(x_q)))\,.
\end{equation}
Remark that $\beta^*(c)$ is $\Gamma$-invariant, vanishes on tuples of points 
that lie in the same horoball in the disjoint union of horoballs $\pi^{-1}(E_i)$,
and is independent of the chosen lift of $\beta'$ 
(but not of the points $b_0,...,b_k$).
Thus, $\beta^*(c)$ is a cochain in $C^q(\mathbb{H}^n,U)^\Gamma$ 
which is now measurable, so that we can integrate it on translates of a given $(q+1)$-tuple of point.
We define
$$
\operatorname{trans}_\beta(c):(\mathbb{H}^n)^{q+1}\longrightarrow \mathbf{R}
$$
by
$$
\operatorname{trans}_\beta(c)(x_0,...,x_q)
:=\int_{\Gamma\backslash G} \varepsilon(\dot g^{-1})\cdot \left( \beta^*(c)(\dot gx_0,...,\dot gx_q)\right)d\mu(\dot g)\,,
$$
where $\mu$ is as in \eqref{trans}.
It is easy to show that the integral is finite. Indeed, let $D$ be the maximum of the distances 
between $x_0$ and $x_i$, for $i=1,...,q$. Then for $\dot g \in \Gamma\backslash G$ 
such that $\dot g x_0$ lies outside a $D$-neighborhood of the compact core $N$, 
each $\dot gx_i$ clearly lies outside $N$ and hence $\beta^*(c)(\dot gx_0,...,\dot gx_q)$ 
vanishes for such $\dot g$. It follows that the integrand vanishes outside
a compact set, within which it takes only finitely many values. 
Furthermore, it follows from the $\Gamma$-invariance of $c$ and $\beta(c)$ 
that $\operatorname{trans}_\beta(c)$ is $G$-invariant.

Since $\operatorname{trans}_\beta$ commutes with the coboundary operator, it induces a cohomology map
$$\operatorname{trans}:H^q(N,\partial N)\longrightarrow H^q_{c}(G,\mathbf{R}_\varepsilon).$$
As the transfer map $\operatorname{trans}_\beta$ restricts to a cochain map 
between the corresponding bounded cocomplexes, 
it also induces a map on the bounded cohomology groups
$$\operatorname{trans}_b:H^q_{b}(N,\partial N)\longrightarrow H^q_{c, b}(G,\mathbf{R}_\varepsilon),$$
and the commutativity of the middle diagram in (\ref{equ: CD with all trans}) is now obvious.

\subsubsection*{Commutativity of the Lower Square}
Denote by 
$\Phi: \Omega^q(\mathbb{H}^n,\mathbf{R}_\varepsilon) \longrightarrow L^\infty((\mathbb{H}^n)^{q+1},\mathbf{R}_\varepsilon)$ 
the map (analogous to the map $\Psi$ defined in \eqref{psi}) sending  the differential form $\alpha$ to the cochain $\Phi(\alpha)$ 
mapping a $(q+1)$-tuple of points $(x_0,...,x_q)\in (\mathbb{H}^n)^{q+1}$ to 
$$\int_{straight(x_0,...,x_q)}\alpha.$$ 
The de Rham isomorphism is realized at the cochain level by precomposing $\Phi$ 
with the map sending a singular simplex in $\mathbb{H}^n$ to its vertices. 
To check the commutativity of the lower square, observe that 
\begin{eqnarray*}
 \operatorname{trans}_\beta\circ \Phi(\alpha)(x_0,...,x_q)&=& 
\int_{\Gamma\backslash G}\varepsilon(\dot g^{-1})\cdot \left( \int_{straight(\beta(\dot gx_0),...,\beta(\dot gx_q))} \alpha \right) d\mu(\dot g)
\end{eqnarray*}
while
\begin{eqnarray*}
 \Phi\circ \operatorname{trans}_{dR}(\alpha)(x_0,...,x_q)&=&
 \int_{\Gamma\backslash G}\varepsilon(\dot g^{-1})\cdot \left( \int_{straight(\dot gx_0,...,\dot gx_q))} \alpha \right) d\mu(\dot g)\,.
\end{eqnarray*}
If $d\alpha=0$, the coboundary of the $G$-invariant cochain 
$$
(x_0,...,x_{q-1})\longmapsto 
\sum_{i=0}^{q-1} (-1)^i \int_{\Gamma\backslash G}
\varepsilon(\dot g^{-1})\cdot \left( \int_{straight(\dot gx_0,...,\dot gx_i,\beta(\dot gx_i),...,\beta(\dot gx_{q-1}))} \alpha \right) d\mu(\dot g)
$$
is equal to the difference of the two given cocycles.

\subsubsection*{Commutativity of the Upper Triangle}
Observe that the isomorphism $H^\bullet_b(M,M\smallsetminus N)\cong H^\bullet_b(\Gamma)$ 
can be induced at the cochain level 
by the map 
$\beta^*: C^q_b(\mathbb{H}^n,U)^\Gamma \rightarrow L^\infty((\mathbb{H}^n)^{q+1},\mathbf{R})^\Gamma$
defined in \eqref{beta-ast} (and for which we allow ourselves a slight abuse of notation).
It is immediate that we now have commutativity of the upper triangle already at the cochain level, 
\begin{equation*}
\xymatrix{ L^\infty((\mathbb{H}^n)^{q+1},\mathbf{R})^\Gamma\ar[rd]^{\operatorname{trans}_\Gamma}&\\
C^q_{b}(\mathbb{H}^n,U)^\Gamma  \ar[u]^{\beta^\ast}\ar[r]_{\operatorname{trans}_b}& L^\infty((\mathbb{H}^n)^{q+1},\mathbf{R}_\varepsilon)^G\,.}
\end{equation*}
This finishes the proof of the proposition. 
\end{proof}

\subsection{Properties of $\mathrm{Vol}(\rho)$}

\begin{lemma}\label{lemma: vol of i} Let $i:\Gamma\hookrightarrow G$
be a lattice embedding. Then 
$$
\mathrm{Vol}(i)=\mathrm{Vol}(M)\,.
$$
\end{lemma}

\begin{proof} Both sides are multiplicative with respect to finite index subgroups. 
We can hence without loss of generality suppose that $\Gamma$ is torsion free. By definition, we have 
\begin{eqnarray*}
\mathrm{Vol}(M)&=& \langle \omega_{N,\partial N} , [N,\partial N] \rangle,\\
\mathrm{Vol}(i)&=& \langle (c \circ i^*)(\omega_n^b), [N,\partial N] \rangle.
\end{eqnarray*}
The desired equality would thus clearly follow from 
$\omega_{N,\partial N}=(c \circ i^*)(\omega_n^b)$. 
As the transfer map $\tau_{dR}:H^n(N,\partial N)\rightarrow H^n_c(G)$ is an isomorphism in top degree and 
sends $\omega_{N,\partial N}$ to $\omega_n$, this is equivalent to
$$
\omega_n
= \tau_{dR}(\omega_{N,\partial N})
=\underbrace{\tau_{dR}\circ c  }_{c \circ \operatorname{trans}_\Gamma}\circ i^*(\omega_n^b)
=c \circ \operatorname{trans}_\Gamma  \circ i^* (\omega_n^b)
=c(\omega_n^b)=\omega_n\,,
$$
where we have used the commutativity of the diagram (\ref{equ: CD with trans}) 
(Proposition~\ref{prop: comm})  
and the fact that $\operatorname{trans}_\Gamma\circ i^*=Id$. \qed
\end{proof}

\begin{proposition}
\label{prop: trans rho}Let $\rho:\Gamma\rightarrow G$
be a representation. The composition 
\begin{equation*}
\xymatrix{ \mathbf{R}\cong H_{c,b}^{n}(G,\mathbf{R}_{\varepsilon}) \ar[r]^-{\rho^\ast}
&H_{b}^{n}(\Gamma) \ar[r]^-{\operatorname{trans}_\Gamma}
&H_{c,b}^{n}(G,\mathbf{R}_{\varepsilon})\cong\mathbf{R} }
\end{equation*}
is equal to $\lambda\cdot \mathrm{Id}$, where
$$|\lambda|=\frac{|\mathrm{Vol}(\rho)|}{\mathrm{Vol}(M)}\leq 1.$$
\end{proposition}

\begin{proof} As the quotient is left invariant by passing to finite index subgroups, 
we can without loss of generality suppose that $\Gamma$ is torsion free. 
Let $\lambda\in \mathbf{R}$ be defined by 
\begin{equation}\label{lambda}
\operatorname{trans}_\Gamma \circ \rho^*(\omega_n^b)=\lambda \cdot \omega_n^b\,.
\end{equation}
We apply the comparison map $c$ to this equality and obtain
$$ 
c \circ \operatorname{trans}_\Gamma \circ \rho^*(\omega_n^b)
=\lambda \cdot c (\omega_n^b)=\lambda \cdot \omega_n=\lambda \cdot \tau_{dR}(\omega_{N,\partial N})\,.
$$
The first expression of this line of equalities is equal to 
$\tau_{dR} \circ c \circ \rho^*(\omega_n^b)$ 
by the commutativity of the diagram (\ref{equ: CD with trans}). Since $\tau_{dR}$ 
is injective in top degree it follows that $(c \circ \rho^*)(\omega_n^b)=\lambda \cdot \omega_{N,\partial N}$. 
Evaluating on the fundamental class, we obtain
$$ 
\mathrm{Vol}(\rho)=\langle  (c \circ \rho^*)(\omega_n^b),[N,\partial N] \rangle 
=\lambda \cdot \langle  \omega_{N,\partial N},[N,\partial N] \rangle =
\lambda \cdot \mathrm{Vol}(i)=\lambda\cdot \mathrm{Vol}(M)\,.
$$
For the inequality, we  take the sup norms on both sides of \eqref{lambda}, and get
$$ 
|\lambda|=\frac{\|\operatorname{trans}_\Gamma \circ \rho^*(\omega_n^b)\|}{\|\omega_n^b\|}\leq 1\,,
$$
where the inequality follows from the fact that all maps involved do not increase the norm.
This finishes the proof of the proposition.\qed
\end{proof}

\section{On the Proof of Theorem~\ref{theorem:maxrep}}
\label{sec:4}

The simple inequality $|\mathrm{Vol}(\rho)|\leq |\mathrm{Vol}(i)|=\mathrm{Vol}(M)$ 
follows from Proposition~\ref{prop: trans rho} and Lemma~\ref{lemma: vol of i}. 

The proof is divided into three steps. The first step, which follows essentially Furstenberg's footsteps
\cite[Chapter 4]{Zimmer_book}, consists in exhibiting a $\rho$-equivariant measurable boundary map 
$\varphi:\partial \mathbb{H}^n\rightarrow \partial \mathbb{H}^n$. 
In the second step we will establish that $\varphi$ maps 
the vertices of almost every positively oriented ideal simplex 
to vertices of positively (or negatively - one or the other, not both) oriented ideal simplices. 
In the third and last step we show that $\varphi$ has to be the extension of an isometry, 
which will provide the conjugation between $\rho$ and $i$. 
The fact that $n\geq 3$ will only be used in the third step. 

\subsection*{Step 1: The Equivariant Boundary Map} We need to define a measurable map 
$\varphi:\partial \mathbb{H}^n\rightarrow \partial \mathbb{H}^n$ such that
\begin{equation}\label{equ:phiequivar}
\varphi(i(\gamma)\cdot \xi)=\rho(\gamma)\cdot \varphi(\xi),
\end{equation}
for every $\xi\in \partial \mathbb{H}^n$ and every $\gamma\in \Gamma$. 

The construction of such boundary map is the sore point of many rigidity questions.  
In the rank one situation in which we are,
the construction is well known and much easier, and is recalled here for completeness.   

Since $\partial\mathbb{H}^n$ can be identified with $\mathrm{Isom}^+(\mathbb{H}^n)/P$, 
where $P<\mathrm{Isom}^+(\mathbb{H}^n)$ is a minimal parabolic,
the action of $\Gamma$ on $\partial\mathbb{H}^n$ is amenable.
Thus there exists a $\Gamma$-equivariant measurable map
$\varphi:\partial \mathbb{H}^n\rightarrow \mathcal M^1(\partial \mathbb{H}^n)$, 
where $\mathcal M^1(\partial \mathbb{H}^n)$ denotes the probability measures on $\partial\mathbb{H}^n$,
\cite{Zimmer_book}.  We recall the proof here for the sake of the reader 
familiar with the notion of amenable group but not conversant
with that of amenable action, although the result is by now classical.  

\begin{lemma}  Let $G$ be a locally compact group, $\Gamma<G$ a lattice and $P$ an amenable subgroup.
Let $X$ be a compact metrizable space with a $\Gamma$-action by homeomorphisms.
Then there exists a $\Gamma$-equivariant boundary map $\varphi:G/P\to\mathcal M^1(X)$.
\end{lemma}

\begin{proof} Let $C(X)$ be the space of continuous functions on $X$.  The space
$$
\begin{aligned}
\mathrm L_\Gamma^1(G,C(X)):=\{f:G\to C(X)\mid \,
f\text{ is measurable, }\Gamma\text{-equivariant and}\\
\int_{\Gamma\backslash G}\|f(\dot g)\|_\infty d\mu(\dot g)<\infty\}\,,
\end{aligned}
$$
is a separable Banach space whose dual is the space 
$\mathrm L_\Gamma^\infty(G,\mathcal M(X))$ of measurable
$\Gamma$-equivariant essentially bounded maps from $G$ into $\mathcal M(X)$, 
where $\mathcal M(X)=C(X)^\ast$ is the dual of $C(X)$.
(Notice that since $C(X)$ is a separable Banach space, 
the concept of measurability of a function $G\to C(X)^*$
is the same as to whether $C(X)^*$ is endowed with the weak-* or the norm topology.) 
Then $\mathrm L_\Gamma^\infty(G,\mathcal M^1(X))$ is a convex compact subset of the unit ball of 
$\mathrm L_\Gamma^\infty(G,\mathcal M(X))$ that is right $P$-invariant.  Since $P$ is amenable,
there exists a $P$-fixed point, that is nothing but the map $\varphi:G/P\to\mathcal M^1(X)$
we were looking for.
\qed
\end{proof}

We are going to associate to every $\mu\in\mathcal M^1(\partial\mathbb{H}^n)$ 
(in the image of $\varphi$) a point in $\partial\mathbb{H}^n$.

If the measure $\mu$ has only one atom of mass $\geq\frac12$, then we associate to $\mu$ this atom.  
We will see that all other possibilities result in a contradiction.

If the measure $\mu$ has no atoms of mass greater than or equal to  $\frac12$, 
we can apply Douady and Earle's barycenter construction \cite[\S~2]{Douady_Earle} 
that to such a measure associates equivariantly a point $b_\mu\in\mathbb{H}^n$.  
By ergodicity of the $\Gamma$-action on $\partial\mathbb{H}^n\times\partial\mathbb{H}^n$,
the distance $d:=d(b_{\varphi(x)},b_{\varphi(x')})$ between any two of these points is essentially constant.  
It follows that for a generic $x\in\partial\mathbb{H}^n$,
there is a bounded orbit, contradicting the non-elementarity of the action.

If on the other hand there is more than one atom whose mass is at least $\frac12$, 
then the support of the measure must consist
of two points (with an equally distributed measure).  
Denote by $g_x$ the geodesic between the two points in the support of the measure 
$\varphi(x)\in\mathcal M^1(\partial\mathbb{H}^n)$.
By ergodicity of the $\Gamma$-action on $\partial\mathbb{H}^n\times\partial\mathbb{H}^{n}$, 
the cardinality of the intersection 
$\operatorname{supp}(\varphi(x))\cap\operatorname{supp}(\varphi(x'))$ 
must be almost everywhere constant and 
hence almost everywhere either equal to $0$, $1$ or $2$.  

If $|\operatorname{supp}(\varphi(x))\cap\operatorname{supp}(\varphi(x'))|=2$ 
for almost all $x,x'\in\partial\mathbb{H}^n$, 
then the geodesic $g_x$ is  $\Gamma$-invariant and hence the action is elementary.  

If $|\operatorname{supp}(\varphi(x))\cap\operatorname{supp}(\varphi(x'))|=1$, 
then we have to distinguish two cases:  
either for almost every $x\in\partial\mathbb{H}^n$ there is a point $\xi\in\partial\mathbb{H}^n$ 
such that $\operatorname{supp}(\varphi(x))\cap\operatorname{supp}(\varphi(x'))=\{\xi\}$ 
for almost all $x'\in \partial\mathbb{H}^n$, 
in which case again $\xi$ would be $\Gamma$-invariant and the action elementary, or 
$\operatorname{supp}(\varphi(x))
\cup\operatorname{supp}(\varphi(x'))\cup\operatorname{supp}(\varphi(x''))$ 
consists of exactly three points for almost
every $x',x''\in\partial\mathbb{H}^n$.  
In this case the barycenter of the geodesic triangle with vertices 
in these three points is $\Gamma$-invariant and the
action is, again, elementary.

Finally, if $|\operatorname{supp}(\varphi(x))\cap\operatorname{supp}(\varphi(x'))|=0$, 
let $D:=d(g_x,g_{x'})$.  By ergodicity on $\partial\mathbb{H}^n\times\partial\mathbb{H}^n$,
$d$ is essentially constant.  Let $\gamma\in\rho(\Gamma)$ 
be a hyperbolic elements whose fixed points are not the endpoints of $g_x$ or $g_{x'}$.
Then iterates of $\gamma$ send a geodesic $g_{x'}$ 
into an arbitrarily small neighborhood of the attractive fixed point of $\gamma$, 
contradicting that $g_x$ is at fixed distance from $g_{x'}$.

\subsection*{Step 2: Mapping Regular Simplices to Regular Simplices}

The next step is to prove Theorem~\ref{thm:formula}.  Then if $\mathrm{Vol}(\rho)=\mathrm{Vol}(M)$,
it will follow that the map $\varphi$ in Step 1 sends almost 
all regular simplices to regular simplices. 

From Proposition~\ref{prop: trans rho} we obtain that the composition of  the induced map $\rho^*$ 
and the transfer with respect to the lattice embedding $i$ is equal to 
$\pm$ the identity on $H^n_{c,b}(G^+,\mathbf{R}_\varepsilon)$. 
In dimension $3$, it follows from \cite{Bloch} that  
$H^3_{c,b}(\mathrm{Isom}^+(\mathbb{H}^3),\mathbf{R})\cong \mathbf{R}$ and 
the proof can be formulated using trivial coefficients; this has been done in \cite{BurgerIozzi06}, 
which is the starting point of this paper.  
In higher dimension it is conjectured, but not known, that $H^n_{c,b}(G^+,\mathbf{R})\cong \mathbf{R}$.

We can without loss of generality suppose that $\operatorname{trans}_\Gamma \circ \rho^*$ is equal to $+Id$. 
Indeed, otherwise, we conjugate $\rho$ by an orientation reversing isometry. 
We will now show that the isomorphism realized at the cochain level, leads to the equality 
(\ref{eq:formula1}), 
which is only an almost everywhere equality. Up to this point, the proof is elementary. 
The only difficulty in our proof is to show that the almost everywhere equality is a true equality, 
which we prove in Proposition~\ref{prop:ae->e}. 
Note however that there are two cases in which Proposition~\ref{prop:ae->e} is immediate, namely 
1) if $\varphi$  is a homeomorphism, 
which is the case if $\Gamma$ is cocompact and $\rho$ is also a lattice embedding
(which is the case of the classical Mostow Rigidity Theorem), and 2) if the dimension $n$ equals $3$. 
We give the alternative simple arguments below. 

The bounded cohomology groups $H^n_{c,b}(G,\mathbf{R}_\varepsilon)$ and 
$H^n_b(\Gamma,\mathbf{R})$ can both be computed from the corresponding $L^\infty$ 
equivariant cochains on $\partial \mathbb{H}^n$. The induced  map 
$\rho^*:H^n_{c,b}(G,\mathbf{R}_\varepsilon)\rightarrow H^n_b(\Gamma,\mathbf{R})$
is represented by the pullback by $\varphi$, 
although it should be noted that the pullback in bounded cohomology cannot be implemented 
with respect to boundary maps in general,
unless the class to pull back can be represented by a strict invariant Borel cocycle
\cite{Burger_Iozzi_app}.  This is our case for $\mathrm{Vol}_n$ and  
as a consequence, $\varphi^*(\mathrm{Vol}_n)$ is also a measurable $\Gamma$-invariant cocycle and 
that determines a cohomology class in $H^n_b(\Gamma)$ (see \cite{BurgerIozzi06}). 
It remains to see that this class is indeed $\rho^*(\omega_n)$. 
In the cocomplex $C(G^{n+1},\mathbf{R}_\varepsilon)^G$, 
the volume class $\omega_n$ is represented by evaluating $\mathrm{Vol}_n$ 
on any point $\xi\in \partial \mathbb{H}^n$,
 thus by a cocycle
$$(g_0,\dots,g_n)\mapsto \mathrm{Vol}_n(g_0\xi,\dots,g_n\xi).$$
In the cocomplex $C(\Gamma^{n+1},\mathbf{R})$, the pull back class  $\rho^*(\omega_n)$ is represented by 
$$(\gamma_0,\dots,\gamma_n)\mapsto \mathrm{Vol}_n(\rho(\gamma_0)\varphi(\xi),\dots,\rho(\gamma_n)\varphi(\xi)).$$
The latter expression is equal to 
$$
\mathrm{Vol}_n(\varphi(\gamma_0 \xi),\dots,\varphi(\gamma_n\xi))=\varphi^*(\mathrm{Vol}_n)(\gamma_0 \xi,\dots,\gamma_n \xi)
$$
for every $\xi \in \partial \mathbb{H}^n$. Thus, evaluation on $\xi$ provides a map 
$L^\infty(\partial \mathbb{H}^n,\mathbf{R})\rightarrow C(\Gamma^{n+1},\mathbf{R})^\Gamma$ 
mapping our a priori unknown cocycle $\varphi^*(\mathrm{Vol}_n)$ to a representative of $\rho^*(\omega_n)$. 
It follows that $\varphi^*(\mathrm{Vol}_n)$ indeed represents $\rho^*(\omega_n)$. 

The composition of maps $\operatorname{trans}_\Gamma \circ \rho^*$ is thus realized at the cochain level by 
$$ \begin{array}{rcl}
L^\infty((\partial \mathbb{H}^n)^{n+1},\mathbf{R}_\varepsilon)^\Gamma &\longrightarrow 
& L^\infty((\partial \mathbb{H}^n)^{n+1},\mathbf{R}_\varepsilon)^G \\
v \qquad\qquad
&\longmapsto 
& \{ (\xi_0,\dots,\xi_n)\mapsto \int_{\Gamma \backslash G} \varepsilon(\dot g^{-1}) v( \varphi(\dot g\xi_0,\dots,\dot g\xi_n))d\mu(\dot g) \}\,.
\end{array}$$
Since the composition  $\operatorname{trans}_\Gamma \circ \rho^*$ is the multiplication by $\frac{\mathrm{Vol}(\rho)}{\mathrm{Vol}(M)}$ 
at the cohomology level and 
there are no coboundaries in degree $n$ (Lemma~\ref{lemma: no coboundary}), 
the above map sends the cocycle  $\mathrm{Vol}_n$ to 
$\frac{\mathrm{Vol}(\rho)}{\mathrm{Vol}(M)}\mathrm{Vol}_n$. 
Thus, for almost every $\xi_0,\dots,\xi_n\in \partial \mathbb{H}^n$ we have 
\begin{equation} \label{eq:formula1} 
\int_{\Gamma\backslash G} \varepsilon(\dot g^{-1})\cdot \mathrm{Vol}_n(\varphi(\dot g\xi_0),\dots,\varphi(\dot g\xi_n))d\mu(\dot g)
=\frac{\mathrm{Vol}(\rho)}{\mathrm{Vol}(M)}\mathrm{Vol}_n(\xi_0,\dots,\xi_n)\,.
\end{equation}
Let $(\partial\mathbb{H}^n)^{(n+1)}$ be the $G$-invariant open subset
of $(\partial\mathbb{H}^n)^{n+1}$ consisting of $(n+1)$-tuples of points 
$(\xi_0,\dots,\xi_{n})$ such that $\xi_i\neq \xi_j$ for all $i\neq j$. 
Observe that the volume cocycle $\mathrm{Vol}_n$ is continuous 
when restricted to $(\partial\mathbb{H}^n)^{(n+1)}$ and vanishes on 
$(\partial\mathbb{H}^n)^{n+1}\smallsetminus(\partial\mathbb{H}^n)^{(n+1)}$.
Observe moreover that the volume of ideal simplices is a continuous extension
of the volume of simplices with vertices in the interior $B^n$ of the sphere $S^{n-1}=\partial\mathbb{H}^n$.

\begin{proposition}\label{prop:ae->e}  Let $i:\Gamma\to G$ be a lattice embedding, 
$\rho:\Gamma\to G$ a representation
and $\varphi:\partial\mathbb{H}^n\to\partial\mathbb{H}^n$ a $\Gamma$-equivariant measurable map.  
Identifying $\Gamma$ with its image $i(\Gamma)<G$ via the lattice embedding, if
\begin{equation}\label{eq:formula} 
\int_{\Gamma\backslash G}\varepsilon(\dot g^{-1})\cdot \mathrm{Vol}_n(\varphi(\dot g\xi_0),\dots,\varphi(\dot g\xi_n))\,d\mu(\dot g)
=\frac{\mathrm{Vol}(\rho)}{\mathrm{Vol}(M)}\mathrm{Vol}_n(\xi_0,\dots,\xi_n)
\end{equation}
for almost every $(\xi_0,\dots,\xi_n)\in(\partial\mathbb{H}^n)^{n+1}$, then the equality holds everywhere.
\end{proposition}

Before we proceed with the proof, 
let us observe that it immediately follows from the proposition that if
$\rho$ has maximal volume, then 
$\varphi$ maps the vertices of almost every regular simplex to the vertices 
of a regular simplex of the same orientation, which is the conclusion of Step 2. 

\begin{proof}[for $\varphi$ homeomorphism]
Since $\varphi$ is injective, both sides of the almost everywhere equality 
are continuous on $(\partial\mathbb{H}^n)^{(n+1)}$. 
Since they agree on a full measure subset of $(\partial\mathbb{H}^n)^{(n+1)}$, 
the equality holds on the whole of $(\partial\mathbb{H}^n)^{(n+1)}$. 
As for its complement, it is clear that  if $\xi_i=\xi_j$ for $i\neq j$ then both sides of the equality vanish.
\qed
\end{proof}

\begin{proof}[for $n=3$] Both sides of the almost equality 
are defined on the whole of $(\partial\mathbb{H}^3)^{4}$, 
are cocycles on the whole of $(\partial\mathbb{H}^3)^{4}$, 
vanish on $(\partial \mathbb{H}^3)^4 \smallsetminus (\partial \mathbb{H}^3)^{(4)}$ 
and are $\mathrm{Isom}^+(\mathbb{H}^3)$-invariant. 
Let $a,b:(\partial\mathbb{H}^3)^{4}\rightarrow \mathbf{R}$ be two such functions and 
suppose that $a=b$ on a set of full measure. This means that for 
{\it almost every} $(\xi_0,...,\xi_3)\in (\partial \mathbb{H}^3)^4$, 
we have $a(\xi_0,...,\xi_3)=b(\xi_0,...,\xi_3)$. 
Since $\mathrm{Isom}^+(\mathbb{H}^3)$ acts transitively on $3$-tuples of distinct points 
in $\mathbb{H}^3$ and both $a$ and $b$ are $\mathrm{Isom}^+(\mathbb{H}^3)$-invariant, 
this means that for {\it every} $(\xi_0,\xi_1,\xi_2)\in (\partial \mathbb{H}^3)^{(3)}$ 
and almost every $\eta \in \partial \mathbb{H}^3$ the equality 
$$
a(\xi_0,\xi_1,\xi_2,\eta)=b(\xi_0,\xi_1,\xi_2,\eta)
$$
holds. Let $\xi_0,...,\xi_3\in \partial\mathbb{H}^3$ be arbitary. 
If $\xi_i=\xi_j$ for $i\neq j$, we have $a(\xi_0,...,\xi_3)=b(\xi_0,...,\xi_3))$ by assumption. 
Suppose $\xi_i\neq \xi_j$ whenever $i\neq j$. By the above, for every $i\in {0,...,3}$ the equality 
$$a(\xi_0,...,\widehat{\xi_i},...,\xi_3,\eta)=b(\xi_0,...,\widehat{\xi_i},...,\xi_3,\eta)$$
holds for $\eta$ in a subset of full measure in $\partial \mathbb{H}^3$.
Let $\eta$ be in the (non empty) intersection 
of these four full measure subsets of $\partial \mathbb{H}^3$. We then have
\begin{eqnarray*}
a(\xi_0,...,\xi_3)&=&\sum_{i=0}^3(-1)^i a(\xi_0,...,\widehat{\xi_i},...,\xi_3,\eta)\\
&=&\sum_{i=0}^3(-1)^i b(\xi_0,...,\widehat{\xi_i},...,\xi_3,\eta)=b(\xi_0,...,\xi_3),
\end{eqnarray*}
where we have used the cocycle relations for $a$ and $b$ in the first and last equality respectively. 
\qed
\end{proof}

\begin{proof}[general case] Observe first of all that for all 
$(\xi_0,\dots,\xi_{n})\in(\partial\mathbb{H}^n)^{n+1}\smallsetminus(\partial\mathbb{H}^n)^{(n+1)}$
the equality holds trivially.

Using the fact that $\partial\mathbb{H}^n\cong S^{n-1}\subset \mathbf{R}^n$, 
let us consider the function $\varphi:\partial\mathbb{H}^n\to\partial\mathbb{H}^n$
as a function $\varphi:\partial\mathbb{H}^n\to\mathbf{R}^n$ and denote by $\varphi_j$, 
for $j=1,\dots,n$ its coordinates.
Since $\partial\mathbb{H}^n\cong G/P$, where $P$ is a minimal parabolic, 
let $\nu$ be the quasi-invariant measure on 
$\partial\mathbb{H}^n$ obtained from the decomposition of the Haar measure $\mu_G$ 
with respect to the Haar measure 
$\mu_P$ on $P$, as in \eqref{eq:reiter}.
According to Lusin's theorem applied to the $\varphi_j$ for $j=1,\dots,n$ 
(see for example \cite[Theorem~2.24]{Rudin}),
for every $\delta>0$ there exist a measurable set $B_{\delta,i}\subset \partial\mathbb{H}^n$ 
with measure $\nu(B_{i,\delta})\leq\delta$
and a continuous function 
$f'_{j,\delta}:\partial\mathbb{H}^n\to\mathbf{R}$ such that $\varphi_j\equiv f'_{j,\delta}$ 
on $\partial\mathbb{H}^n\smallsetminus B_{j,\delta}$.
Set $ f'_\delta: =( f_{1,\delta},\dots, f_{n,\delta})\to\mathbf{R}^n$ and consider the composition 
$ f_\delta:=r\circ f'_\delta$ with the retraction $r:\mathbf{R}^n\to\overline{B^n}$
to the closed unit ball $\overline{B^n}\subset\mathbf{R}^n$.  
Then, by setting $B_\delta:=\cup_{j=1}^{n}B_{j,\delta}$, $\varphi$ 
coincides on $\partial\mathbb{H}^n\smallsetminus B_\delta$  
with the continuous function $ f_\delta:\partial\mathbb{H}^n\to \overline{B^n}$ 
and $\nu(B_\delta)\leq n\delta$. 

Let $\mathcal D\subset G$ be a fundamental domain for the action of $\Gamma$ on $G$.
For every measurable subset $E\subset\mathcal D$, 
any measurable map $\psi:\partial\mathbb{H}^n\to\overline{B^n}$ and any 
point $(\xi_0,\dots,\xi_n)\in(\partial\mathbb{H}^n)^{(n+1)}$, we use the notation
\begin{equation*}
\mathcal I(\psi,E,(\xi_0,\dots,\xi_n)):=\int_E\varepsilon(g^{-1})\mathrm{Vol}_n(\psi(g\xi_0),\dots,\psi(g\xi_n))\,d\mu_G(g)\,,
\end{equation*}
so that we need to show that if 
\begin{equation}\label{eq:formulaD} 
\mathcal I(\varphi,\mathcal D,(\xi_0,\dots,\xi_n))=\frac{\mathrm{Vol}(\rho)}{\mathrm{Vol}(M)}\mathrm{Vol}_n(\xi_0,\dots,\xi_n)
\end{equation}
for almost every $(\xi_0,\dots,\xi_n)\in(\partial\mathbb{H}^n)^{(n+1)}$, then the equality holds everywhere.

Fix $\epsilon >0$ and let $K_\epsilon\subset\mathcal D$ be a compact set 
such that $\mu_G(\mathcal D\smallsetminus K_\epsilon)<\epsilon$.
The proof is broken up in several lemmas, that we state and use here, but whose proof we postpone.

\begin{lemma}\label{lemma:measurekepbdelta} With the above notations,
\begin{equation}\label{eq:measurekepbdelta}
\mu_G(\{g\in K_\epsilon:\, g\xi\in B_\delta\})\leq \sigma_\epsilon(\delta)\,,
\end{equation}
where $\sigma_\epsilon(\delta)$ does not depend on $\xi\in\partial\mathbb{H}^n$ and
$\sigma_\epsilon(\delta)\to0$ when $\delta\to0$.
\end{lemma}

Replacing $\varphi$ with $f_\delta$ results in the following estimate for the integral.  

\begin{lemma}\label{lemma:kep-bdelta}  With the notation as above, 
there exists a function $M_\epsilon(\delta)$
with the property that $\lim_{\delta\to0}M_\epsilon(\delta)=0$, such that
\begin{equation*}
\left| \mathcal I(\varphi,K_\epsilon,(\xi_0,\dots,\xi_n))-\mathcal I(f_\delta,K_\epsilon,(\xi_0,\dots,\xi_n))\right|\leq M_\epsilon(\delta)\,,
\end{equation*}
for all $(\xi_0,\dots,\xi_{n+1})\in(\partial\mathbb{H}^n)^{n+1}$.
\end{lemma}

Observe that, although
\begin{equation}\label{eq:kep}
\left|\mathcal I(\varphi,\mathcal D,(\xi_0,\dots,\xi_n))-\mathcal I(\varphi,K_\epsilon,(\xi_0,\dots,\xi_n))\right|<\epsilon\|\mathrm{Vol}_n\|\,,
\end{equation}
for all $(\xi_0,\dots,\xi_{n+1})\in(\partial\mathbb{H}^n)^{(n+1)}$, the estimate
\begin{equation}\label{eq:F1F3}
\begin{aligned}
      & \left|\mathcal I(\varphi,K_\epsilon,(\xi_0,\dots,\xi_n)) -\frac{\mathrm{Vol}(\rho)}{\mathrm{Vol}(M)}\mathrm{Vol}_n(\xi_0,\dots,\xi_n)\right|\\
\leq&\left|\mathcal I(\varphi,K_\epsilon,(\xi_0,\dots,\xi_n))-\mathcal I(\varphi,\mathcal D,(\xi_0,\dots,\xi_n)) \right|\\
&\hphantom{XXX}+\left| \mathcal I(\varphi,\mathcal D,(\xi_0,\dots,\xi_n))-\frac{\mathrm{Vol}(\rho)}{\mathrm{Vol}(M)}\mathrm{Vol}_n(\xi_0,\dots,\xi_n)\right|
\leq \epsilon\|\mathrm{Vol}_n\|\,,
\end{aligned}
\end{equation}
holds only for almost every $(\xi_0,\dots,\xi_{n})\in(\partial\mathbb{H}^n)^{(n+1)}$, 
since this is the case for \eqref{eq:formulaD}.

From \eqref{eq:F1F3} and Lemma~\ref{lemma:kep-bdelta}, it follows that 
\begin{equation}\label{eq:F2F3ae}
\begin{aligned}
      & \left|\mathcal I(f_\delta,K_\epsilon,(\xi_0,\dots,\xi_n))-\frac{\mathrm{Vol}(\rho)}{\mathrm{Vol}(M)}\mathrm{Vol}_n(\xi_0,\dots,\xi_n)\right|\\
\leq&\left|\mathcal I(f_\delta,K_\epsilon,(\xi_0,\dots,\xi_n))-\mathcal I(\varphi,K_\epsilon,(\xi_0,\dots,\xi_n))\right|\\
&\hphantom{XXX}+\left|\mathcal I(\varphi,K_\epsilon,(\xi_0,\dots,\xi_n)) -\frac{\mathrm{Vol}(\rho)}{\mathrm{Vol}(M)}\mathrm{Vol}_n(\xi_0,\dots,\xi_n)\right|\\
	<& M_\epsilon(\delta)+\epsilon\|\mathrm{Vol}_n\|\,,
\end{aligned}
\end{equation}
for almost every $(\xi_0,\dots,\xi_{n})\in(\partial\mathbb{H}^n)^{(n+1)}$.

The following lemma uses the continuity of $f_\delta$ to deduce that 
all of the almost everywhere equality that propagated from 
the use of \eqref{eq:formulaD} in \eqref{eq:F1F3}, 
can indeed be observed to hold everywhere because of the use of Lusin theorem.

\begin{lemma}\label{lemma:continuity} There exist a function $L(\epsilon,\delta)$ 
such that $\lim_{\epsilon\to0}\lim_{\delta\to0} L(\epsilon,\delta)=0$ and
\begin{equation}\label{eq:cont}
|\mathcal I(f_\delta,K_\epsilon,(\xi_0,\dots,\xi_n))-\frac{\mathrm{Vol}(\rho)}{\mathrm{Vol}(M)}\mathrm{Vol}_n(\xi_0,\dots,\xi_n)|
\leq L(\epsilon,\delta)
\end{equation}
for all $(\xi_0,\dots,\xi_n)\in(\partial\mathbb{H}^n)^{(n+1)}$.
\end{lemma}

From this, and from Lemma~\ref{lemma:kep-bdelta}, and using once again \eqref{eq:kep}, 
now all everywhere statements,
we conclude that 
\begin{equation*}
\begin{aligned}
	&\left|\mathcal I(\varphi,\mathcal D,(\xi_0,\dots,\xi_n))-\frac{\mathrm{Vol}(\rho)}{\mathrm{Vol}(M)}\mathrm{Vol}_n(\xi_0,\dots,\xi_n)\right|\\
  \leq&\left|\mathcal I(\varphi,\mathcal D,(\xi_0,\dots,\xi_n))-\mathcal I(\varphi,K_\epsilon,(\xi_0,\dots,\xi_n))\right|\\
  	&\hphantom{XXX}+\left| \mathcal I(\varphi,K_\epsilon,(\xi_0,\dots,\xi_n))-\mathcal I(f_\delta,K_\epsilon,(\xi_0,\dots,\xi_n))\right|\\
	&\hphantom{XXX}+\left|\mathcal I(f_\delta,K_\epsilon,(\xi_0,\dots,\xi_n))-\frac{\mathrm{Vol}(\rho)}{\mathrm{Vol}(M)}\mathrm{Vol}_n(\xi_0,\dots,\xi_n)\right|\\
      <&M_\epsilon(\delta)+L(\epsilon,\delta)+\epsilon\|\mathrm{Vol}_n\|\,,
\end{aligned}
\end{equation*}
for all $(\xi_0,\dots,\xi_{n+1})\in(\partial\mathbb{H}^n)^{n+1}$. 
 This concludes the proof of Proposition~\ref{prop:ae->e}, assuming the unproven lemmas.
\qed
\end{proof}

We now proceed to the proof of Lemmas~\ref{lemma:measurekepbdelta},~\ref{lemma:kep-bdelta} 
and~\ref{lemma:continuity}.
\begin{proof}[Proof of Lemma~\ref{lemma:measurekepbdelta}]
Recall that $\partial\mathbb{H}^n=G/P$, where $P<G$ is a minimal parabolic
and let $\eta:G/P\to G$ be a Borel section of the projection $G\to G/P$ 
such that $F:=\eta(G/P)$ is relatively compact \cite[Lemma~1.1]{Mackey}
Let $\tilde B_\delta:=\eta(B_\delta)$ and, if $\xi\in B_\delta$, 
set $\tilde \xi:=\eta(\xi)\in\tilde B_\delta$.  On the other hand, 
if $g\in K_\epsilon$ and $g\xi\in B_\delta$,
there exists $p\in P$ such that $g\tilde \xi p\in \tilde B_\delta$ and, 
in fact, the $p$ can be chosen to be in $P\cap F^{-1}(K_\epsilon)^{-1}F=:C_\epsilon$.
Thus we have
\begin{equation*}
\begin{aligned}
   &\{g\in K_\epsilon:\,g\xi\in B_\delta\}
   =\{g\in K_\epsilon:\,\text{there exists }p\in C_\epsilon\text{ with }g\tilde \xi p\in \tilde B_\delta\}\\
=&\{g\in K_\epsilon\cap\tilde B_\delta p^{-1}\tilde \xi^{-1}\text{ for some }p\in C_\epsilon\}
\subset K_\epsilon\cap\tilde B_\delta C_\epsilon^{-1}\tilde \xi^{-1}\,,
\end{aligned}
\end{equation*}
and hence
\begin{equation*}
\mu_G(\{g\in K_\epsilon:\, g\xi\in B_\delta\})\leq\mu_G(K_\epsilon\tilde \xi^{-1}\cap\tilde B_\delta C_\epsilon^{-1})\leq\mu_G(\tilde B_\delta C_\epsilon^{-1})\,.
\end{equation*}

To estimate the measure, recall that 
there is a strictly positive continuous function $q:G\to\mathbf{R}^+$ 
and a positive measure $\nu$ on $\partial\mathbb{H}^n$
such that 
\begin{equation}\label{eq:reiter}
\int_G f(g)q(g)\,d\mu_G(g)=\int_{\partial\mathbb{H}^n}\left(\int_P f(\dot g\xi)\,d\mu_P(\xi)\right) d\nu(\dot g)\,,
\end{equation}
for all continuous functions $f$ on $G$ with compact support, 
\cite[\S.8.1]{Reiter_Stegeman}.

We may assume that $\mu_G(\tilde B_\delta C_\epsilon^{-1})\neq0$ (otherwise we are done).  
Then, since $q$ is continuous and strictly positive and the integral is on a relatively compact set, 
there exists a constant $0<\alpha<\infty$ such that
\begin{equation}\label{eq:reiter2}
\alpha\mu_G(\tilde B_\delta C_\epsilon^{-1})=\int_{\partial\mathbb{H}^n}\left(\int_P \chi_{\tilde B_\delta C_\epsilon^{-1}}(\dot g\xi)\,d\mu_P(\xi)\right) d\nu(\dot g)\,.
\end{equation}
But, by construction, if $g\in \tilde B_\delta$, then $g\xi\in\tilde B_\delta C_\epsilon^{-1}$ 
if and only if $\xi\in C_\epsilon^{-1}$, so that
\begin{equation*}
\int_P \chi_{\tilde B_\delta C_\epsilon^{-1}}(\dot g\xi)\,d\mu_P(\xi)=\mu_P(C_\epsilon^{-1})\,,
\end{equation*}
and hence 
\begin{equation*}
\alpha\mu_G(\tilde B_\delta C_\epsilon^{-1})=\nu(B_\delta)\mu_P(C_\epsilon^{-1})\,.
\end{equation*}
Since $\nu(B_\delta)<\delta$, 
the inequality \eqref{eq:measurekepbdelta} is proven with 
$\sigma_\epsilon(\delta)=\frac1\alpha\mu_P(C_\epsilon^{-1})\delta$.
\qed
\end{proof}

\begin{proof}[Proof of Lemma~\ref{lemma:kep-bdelta}]  Let us fix
$(\xi_0,\dots,\xi_n)\in(\partial\mathbb{H}^n)^{n+1}$.  Then we have
\begin{equation*}
\begin{aligned}
       &\left| \mathcal I(\varphi,K_\epsilon,(\xi_0,\dots,\xi_n))-\mathcal I(f_\delta,K_\epsilon,(\xi_0,\dots,\xi_n))\right|\\
 \leq&\left| \mathcal I(\varphi,{K_{\epsilon,0}},(\xi_0,\dots,\xi_n))-\mathcal I(f_\delta,{K_{\epsilon,0}},(\xi_0,\dots,\xi_n))\right|\\
        &\hphantom{XXX}+\left| \mathcal I(\varphi,{K_{\epsilon,1}},(\xi_0,\dots,\xi_n))-\mathcal I(f_\delta,{K_{\epsilon,1}},(\xi_0,\dots,\xi_n))\right|\,,
\end{aligned}
\end{equation*}
where 
\begin{equation*}
{K_{\epsilon,0}}:=\bigcap_{j=0}^n\{g\in K_\epsilon:\,g\xi_j\in\partial\mathbb{H}^n\smallsetminus B_\delta\}\quad\text{ and }\quad
{K_{\epsilon,1}}:=K_\epsilon\smallsetminus{K_{\epsilon,0}}\,.
\end{equation*}
But $\varphi(g)=f_\delta(g)$ for all $g\in{K_{\epsilon,0}}$, and hence difference of the integrals on ${K_{\epsilon,0}}$ vanishes.
Since
\begin{equation*}
\mu_G({K_{\epsilon,1}})=\mu_G\left(K_\epsilon\cap\bigcup_{j=0}^n\{g\in K_\epsilon:\,g\xi_j\in B_\delta\}\right)\leq (n+1)\sigma_\epsilon(\delta)\,,
\end{equation*}
we obtain the assertion with $M_\epsilon(\delta):=2(n+1)\|\mathrm{Vol}_n\| \sigma_\epsilon(\delta)$.
\qed
\end{proof}

\begin{proof}[Proof of Lemma~\ref{lemma:continuity}]
If the volume were continuous on $(\partial\mathbb{H}^n)^{n+1}$ or if the function $f_\delta$
were injective, the assertion would be obvious. 

Observe that $\varphi$ is almost everywhere injective:  in fact, by double ergodicity, 
the subset of $\partial\mathbb{H}^n\times\partial\mathbb{H}^n$ consisting of pairs $(x,y)$ 
for which $\varphi(x)=\varphi(y)$ is a set of either
zero or full measure and the latter would contradict elementarity of the action.
Then on a set of full measure in $\partial\mathbb{H}^n\smallsetminus B_\delta$ the function $f_\delta$ is
injective and hence $\mathrm{Vol}_n(f_\delta(g\xi_0),\dots,f_\delta(g\xi_n))$ is continuous provided
the $f_\delta(g\xi_0),\dots,f_\delta(g\xi_n)$ are pairwise distinct.

So, for any $(\xi_0,\dots,\xi_n)\in(\partial\mathbb{H}^n)^{(n+1)}$ we define
\begin{equation*}
\mathcal E(\xi_0,\dots,\xi_n):=\{g\in K_\epsilon:\,f_\delta(g\xi_0),\dots,f_\delta(g\xi_n)\text{ are pairwise distinct}\}\,.
\end{equation*}
Let $F\subset(B_\delta^c\times B_\delta^c)^{(2)}$ be the set of distinct pairs 
$(x,y)\in (B_\delta^c\times B_\delta^c)^{(2)}$ such that $f_\delta(x)=f_\delta(y)$.
Then $F$ is of measure zero, and given any 
$(\xi_0,\xi_1)\in\partial\mathbb{H}^n\times\partial\mathbb{H}^n$ distinct,
the set $\{g\in G:\,g(\xi_0,\xi_1)\in F\}$ is of $\mu_G$-measure zero. 
This, together with Lemma~\ref{lemma:measurekepbdelta}, implies that 
\begin{equation}\label{eq:measurekepex}
\mu_G(K_\epsilon\smallsetminus\mathcal E(\xi_0,\dots,\xi_n))\leq\mu_G\left(\bigcup_{j=0}^n\{g\in K_\epsilon:\,g\xi_j\in B_\delta\}\right)\leq(n+1)\sigma_\epsilon(\delta)\,.
\end{equation}

Let $\mathcal S\subset(\partial\mathbb{H}^n)^{(n+1)}$ be the set of full measure where \eqref{eq:F2F3ae} 
holds and let $(\xi_0,\dots,\xi_n)\in(\partial\mathbb{H}^n)^{(n+1)}$.
Since $\nu^{n+1}(\partial\mathbb{H}^n)^{(n+1)}\smallsetminus\mathcal S)=0$, 
there exists a sequence of points $(\xi_0^{(k)},\dots,\xi_n^{(k)})\in\mathcal S$ 
with $(\xi_0^{(k)},\dots,\xi_n^{(k)})\to (\xi_0,\dots,\xi_n)$.   Then for every $g\in \mathcal E(\xi)$
\begin{equation*}
\lim_{k\to\infty}\mathrm{Vol}_n(f_\delta(g\xi_0^{(k)}),\dots,f_\delta(g\xi_n^{(k)}))=\mathrm{Vol}_n(f_\delta(g\xi_0),\dots,f_\delta(g\xi_n))\,,
\end{equation*}
and, by the Dominated Convergence Theorem applied to the sequence $h_k(g):=\mathrm{Vol}_n(f_\delta(g\xi_0^{(k)}),\dots,f_\delta(g\xi_n^{(k)}))$, we deduce that
\begin{equation}\label{eq:DCT}
\lim_{k\to\infty}\mathcal I(f_\delta,\mathcal E(\xi_0,\dots,\xi_n),(\xi_0^{(k)},\dots,\xi_n^{(k)}))=\mathcal I(f_\delta,\mathcal E(\xi_0,\dots,\xi_n),(\xi_0,\dots,\xi_n))\,.
\end{equation}
But then
\begin{equation*}
\begin{aligned}
       &|\mathcal I(f_\delta,K_\epsilon,(\xi_0,\dots,\xi_n))-\frac{\mathrm{Vol}(\rho)}{\mathrm{Vol}(M)}\mathrm{Vol}_n(\xi_0,\dots,\xi_n)|\\
\leq&\left|\mathcal I(f_\delta,K_\epsilon,(\xi_0,\dots,\xi_n))-\mathcal I(f_\delta,\mathcal E(\xi_0,\dots,\xi_n),\xi_0,\dots,\xi_n)) \right|\\
    &\hphantom{XXX}+\left|\mathcal I(f_\delta,\mathcal E(\xi_0,\dots,\xi_n),(\xi_0,\dots,\xi_n))-\mathcal I(f_\delta,\mathcal E(\xi_0,\dots,\xi_n),(\xi_0^{(k)},\dots,\xi_n^{(k)}))\right|\\
    &\hphantom{XXX}+\left|\mathcal I(f_\delta,\mathcal E(\xi_0,\dots,\xi_n),(\xi_0^{(k)},\dots,\xi_n^{(k)}))-\mathcal I(f_\delta,K_\epsilon,(\xi_0^{(k)},\dots,\xi_n^{(k)}))\right|\\
    &\hphantom{XXX}+\left|\mathcal I(f_\delta,K_\epsilon,(\xi_0^{(k)},\dots,\xi_n^{(k)}))-\frac{\mathrm{Vol}(\rho)}{\mathrm{Vol}(M)}\mathrm{Vol}_n(\xi_0^{(k)},\dots,\xi_n^{(k)})\right|\\
    &\hphantom{XXX}+\left|\frac{\mathrm{Vol}(\rho)}{\mathrm{Vol}(M)}\mathrm{Vol}_n(\xi_0^{(k)},\dots,\xi_n^{(k)})-\frac{\mathrm{Vol}(\rho)}{\mathrm{Vol}(M)}\mathrm{Vol}_n(\xi_0,\dots,\xi_n)\right|\,,
\end{aligned}
\end{equation*} 
for all $(\xi_0,\dots,\xi_n)\in(\partial\mathbb{H}^n)^{(n+1)}$.

The first and third line after the inequality sign are each 
$\leq(n+1)\|\mathrm{Vol}_n\|\sigma_\epsilon(\delta)$ because of \eqref{eq:measurekepex};
the second line after the equality is less than $\delta$ if $k$ is large enough 
because of \eqref{eq:DCT};  
the fourth line is $\leq M_\epsilon(\delta)+\epsilon\|\mathrm{Vol}_n\|$ by \eqref{eq:F2F3ae} 
since $(\xi_0^{(k)},\dots,\xi_n^{(k)})\in\mathcal S$ and 
finally the last line is also less than $\delta$ if $k$ if large enough.
All of the estimate hold for all $(\xi_0,\dots,\xi_n)\in(\partial\mathbb{H}^n)^{(n+1)}$, 
and hence the assertion is proven
with 
$L(\epsilon,\delta):=
2\delta+2(n+1)\|\mathrm{Vol}_n\|\sigma_\epsilon(\delta)+M_\epsilon(\delta)+\epsilon\|\mathrm{Vol}_n\|$.
\qed
\end{proof}

\subsection*{Step 3: The Boundary Map is an Isometry}
Suppose now that the equality $|\mathrm{Vol}(\rho)|=|\mathrm{Vol}(i)|$ holds.
Then $\varphi$ maps enough regular simplices to regular simplices.
In this last step of the proof, we want to show that
then $\varphi$ is essentially an isometry, 
and this isometry will realize the conjugation between $\rho$ and $i$. 

In the case of a cocompact lattice $\Gamma< \mathrm{Isom}(\mathbb{H}^n)$ 
and a lattice embedding $\rho:\Gamma\rightarrow \mathrm{Isom}^+(\mathbb{H}^n)$, 
the limit map $\varphi$ is continuous and the proof is very simple based on Lemma~\ref{lemma: step 3}. 
This is the original setting of Gromov's proof of Mostow rigidity for compact hyperbolic manifolds. 

If either the representation $\rho$ is not assumed to be a lattice embedding, 
or if $\Gamma$ is not cocompact, 
then the limit map $\varphi$ is only measurable and 
one needs a measurable variant of Lemma~\ref{lemma: step 3} 
presented in Proposition~\ref{prop: step 3} for $n\geq 4$. 
The case $n=3$ was first proven by Thurston for his generalization 
(Corollary~\ref{thm: Thurston Gromov Mostow} here) of Gromov's proof of Mostow rigidity. 
It is largely admitted that the case $n=3$ easily generalizes to $n\geq 4$, 
although we wish to point out that the proof is very much simpler for $n\geq 4$ based on the fact that 
the reflection group of a regular simplex is dense in the isometry group. 
For the proof of Propostion~\ref{prop: step 3}, we will omit the case $n=3$ 
which is nicely written down in all necessary details by Dunfield \cite[pp. 654-656]{Dunfield}, 
following the original \cite[two last paragraphs of Section 6.4]{Thurston_notes}. 

Let $T$ denote the set of $(n+1)$-tuples of points in $\partial \mathbb{H}^n$ 
which are vertices of a regular simplex,
$$
T=\{ \underline{\xi}=(\xi_0,...,\xi_n)\in (\partial \mathbb{H}^n)^{n+1}\mid 
\underline{\xi} \mathrm{ \ are \ the \ vertices \ of \ an \ ideal \ regular \ simplex} \}\,.
$$
We will call an $(n+1)$-tuple in $T$ a regular simplex. 
Note that the order of the vertices $\xi_0,...,\xi_n$ 
induces an orientation on the simplex $\underline{\xi}$. For $\underline{\xi}\in T$, 
denote by $\Lambda_{\underline{\xi}}<\mathrm{Isom}(\mathbb{H}^n)$ the reflection group 
generated by the reflections in the faces of the simplex $\underline{\xi}$.

\begin{lemma} \label{lemma: step 3} Let $n\geq 3$. Let $\underline{\xi}=(\xi_0,...,\xi_n) \in T$. 
Suppose that $\varphi:\partial \mathbb{H}^n\rightarrow \partial \mathbb{H}^n$ is a map 
such that for every $\gamma\in \Lambda_{\underline{\xi}}$, 
the simplex with vertices $(\varphi(\gamma \xi_0),...,\varphi(\gamma \xi_n))$ is regular 
and of the same orientation as $(\gamma \xi_0,...,\gamma \xi_n)\in T$. 
Then there exists a unique isometry $h\in \mathrm{Isom}(\mathbb{H}^n)$ 
such that $h(\xi)=\varphi(\xi)$ for every $\xi\in \cup_{i=0}^n \Lambda_{\underline{\xi}} \xi_i$. 
\end{lemma}

Note that this lemma and its subsequent proposition are the only places in the proof 
where the assumption $n\geq 3$ is needed. 
The lemma is wrong for $n=2$ since $\varphi$ could be any orientation preserving homeomorphism 
of $\partial \mathbb{H}^2$.

\begin{proof} If $\underline{\xi}=(\xi_0,...,\xi_n)$ and $(\varphi(\xi_0),...,\varphi(\xi_n))$ 
belong to $T$, then there exists a unique isometry $h\in \mathrm{Isom}^+(\mathbb{H}^n)$ 
such that $h\xi_i=\varphi(\xi_i)$ for $i=0,...,n$. It remains to check that 
\begin{equation}\label{equ: h=phi} h(\gamma \xi_i)=\varphi(\gamma \xi_i)\end{equation}
for every $\gamma\in \Lambda_{\underline{\xi}}$. 
Every $\gamma\in \Lambda_{\underline{\xi}}$ is a product $\gamma=r_k\cdot ... \cdot r_1$, 
where $r_j$ is a reflection in a face of the regular simplex $r_{j-1}\cdot ... \cdot r_1(\underline{\xi})$. 
We prove the equality (\ref{equ: h=phi}) by induction on $k$, the case $k=0$ being true by assumption. 
Set $\eta_i=r_{k-1}\cdot ... \cdot r_1(\xi_i)$. By induction, we know that $h(\eta_i)=\varphi(\eta_i)$. 
We need to show that $h(r_k\eta_i)=\varphi(r_k\eta_i)$. 
The simplex $(\eta_0,...,\eta_n)$ is regular and $r_k$ is a reflection in one of its faces, 
say the face containing $\eta_1,...,\eta_n$. Since $r_k\eta_i=\eta_i$ for $i=1,...,n$, 
it just remains to show that $h(r_k\eta_0)=\varphi(r_k\eta_0)$. 
The simplex $(r_k\eta_0,r_k\eta_1,...,r_k\eta_n)=(r_k\eta_0,\eta_1,...,\eta_n)$ 
is regular with opposite orientation to $(\eta_0,\eta_1,...,\eta_n)$. 
This implies on the one hand that the simplex $(h(r_k\eta_0),h(\eta_1),...,h(\eta_n))$ 
is regular with opposite orientation to $(h(\eta_0),h(\eta_1),...,h(\eta_n))$, 
and on the other hand that the simplex $(\varphi(r_k \eta_0),\varphi(\eta_1),...,\varphi(\eta_n))$ 
is regular with opposite orientation to $(\varphi(\eta_0),...,\varphi(\eta_n))$. 
Since  $(h(\eta_0),h(\eta_1),...,h(\eta_n))=(\varphi(\eta_0),...,\varphi(\eta_n))$ 
and there is in dimension $n\geq 3$ only one regular simplex with face $h(\eta_1),...,h(\eta_n)$ 
and opposite orientation to $(h(\eta_0),h(\eta_1),...,h(\eta_n))$ 
it follows that $h(r_k\eta_0)=\varphi(r_k\eta_0)$.\qed
 \end{proof} 

If $\varphi$ were continuous, 
sending the vertices of all positively (respectively negatively) oriented ideal regular simplices 
to vertices of positively (resp. neg.) oriented ideal regular simplices, 
then it would immediately follow from the lemma 
that $\varphi$ is equal to an isometry $h$ 
on the orbits $\cup_{i=0}^n \Lambda_{\underline{\xi}} \xi_i$ of the vertices of one regular simplex 
under its reflection group. 
Since the set $\cup_{i=0}^n \Lambda_{\underline{\xi}} \xi_i$ is dense in $\partial \mathbb{H}^n$, 
the continuity of $\varphi$ would imply that $\varphi$ is equal to the isometry $h$ 
on the whole $\partial \mathbb{H}^n$.

In the setting of the next proposition, 
we first need to show that there exist enough regular simplices 
for which $\varphi$ maps every simplex of its orbit under reflections to a regular simplex. 
Second, we apply the lemma to obtain that $\varphi$ is equal to an isometry on these orbits. 
Finally, we use ergodicity of the reflection groups to conclude that 
it is the same isometry for almost all regular simplices. 
As mentioned earlier, the proposition also holds for $n=3$ (see \cite[pp. 654-656]{Dunfield} and 
\cite[two last paragraphs of Section 6.4]{Thurston_notes}), but in that case the proof is quite harder, 
since the reflection group of a regular simplex is discrete in $\mathrm{Isom}(\mathbb{H}^{n})$ 
(indeed, one can tile $\mathbb{H}^3$ by regular ideal simplices) 
and in particular does not act ergodically on $\mathrm{Isom}(\mathbb{H}^{n})$. 

\begin{proposition} \label{prop: step 3} Let $n\geq 4$.
Let $\varphi:\partial \mathbb{H}^n\rightarrow \partial \mathbb{H}^n$ be a measurable map 
sending the vertices of almost every positively, 
respectively negatively oriented regular ideal simplex to the vertices of a positively, resp. negatively, 
oriented regular ideal simplex. Then $\varphi$ is equal almost everywhere to an isometry. 
\end{proposition} 

\begin{proof} Let $T^\varphi \subset T$ denote the following subset of the set $T$ of regular simplices: 
$$T^\varphi= \left\{ \left. \underline{\xi}=(\xi_0,...,\xi_n)\in T \right| 
\begin{array}{cc}
(\varphi(\xi_0),...,\varphi(\xi_n)) \mathrm{\ belongs \ to \ } T \\
\mathrm{\ and \ has \ the \ same \ orientation \ as \ } (\xi_0,...,\xi_n)
\end{array}
\right\}.$$
By assumption, $T^\varphi$ has full measure in $T$. 
Let $T^\varphi_\Lambda\subset T^\varphi$ be the subset consisting of those regular simplices 
for which all reflections by the reflection group $\Lambda_{\underline{\xi}}$ are in $T^\varphi$,
$$
T^\varphi_\Lambda= \{\underline{\xi} \in T \mid \gamma \underline{\xi}\in T^\varphi \ \forall \gamma\in \Lambda_{\underline{\xi}}\}\,.
$$
We claim that $T^\varphi_\Lambda$ has full measure in $T$. 

To prove the claim, we do the following identification. 
Since $G=\mathrm{Isom}(\mathbb{H}^n)$ acts simply transitively on the set $T$ of (oriented) regular simplices, 
given a base point $\underline{\eta}=(\eta_0,...,\eta_n)\in T$ we can identify $G$ with $T$ via the evaluation map
$$
\begin{array}{rcl}
Ev_{\eta}: G&\longrightarrow &T \\
g &\longmapsto &g(\underline{\eta})\,.
\end{array}
$$
The subset $T^\varphi$ is mapped to a subset $G^\varphi:= (Ev_{\eta})^{-1}(T^\varphi) \subset G$ 
via this correspondence. 
A regular simplex $\underline{\xi}=g(\underline{\eta})$ belongs to $T^\varphi_\Lambda$ if and only if, by definition,
$\gamma \underline{\xi}=\gamma g \underline{\eta}$ belongs to $T^\varphi$ for every
$\gamma\in \Lambda_{\underline{\xi}}$. 
Since $ \Lambda_{\underline{\xi}}=g \Lambda_{\underline{\eta}} g^{-1}$, 
the latter condition is equivalent to $ g \gamma_0  \underline{\eta} \in T^\varphi$ 
for every $\gamma_0\in \Lambda_{\underline{\eta}} $, or in other words, $g\in G^\varphi \gamma_0^{-1}$. 
The subset  $T^\varphi_\Lambda$ is thus mapped to 
$$ 
G^\varphi=Ev_{\underline{\eta}}^{-1}(T_\Lambda^\varphi)
=\cap_{\gamma_0\in \Lambda_{\underline{\eta}}} G^\varphi \gamma_0^{-1} \subset G
$$
via the above correspondence. Since a countable intersection of full measure subsets has full measure, 
the claim is proved. 

For every $\underline{\xi}\in T^\varphi_\Lambda$ and hence almost every  $\underline{\xi}\in T$ 
there exists by Lemma~\ref{lemma: step 3} a unique isometry $h_{\underline{\xi}}$ 
such that $h_{\underline{\xi}}(\xi)=\varphi(\xi) $ on the orbit points  
$\xi \in  \cup_{i=0}^n \Lambda_{\underline{\xi}} \xi_i$. 
By the uniqueness of the isometry, 
it is immediate that $h_{\gamma \underline{\xi}}=h_{\underline{\xi}}$ for every $\gamma \in  \Lambda_{\underline{\xi}}$.
We have thus a map $h:T\rightarrow \mathrm{Isom}(\mathbb{H}^n)$
 given by $\underline{\xi}\mapsto h_{\underline{\xi}}$ defined on a full measure subset of $T$. 
 Precomposing $h$ by $Ev_{\underline{\eta}}$, 
 it is straightforward that the left $ \Lambda_{\underline{\xi}}$-invariance of $h$ on  
 $ \Lambda_{\underline{\xi}} \underline{\xi}$ 
 naturally translates to a global right invariance of $h\circ Ev_{\underline{\eta}}$ on $G$. 
Indeed, let $g \in G$ and 
 $ \gamma_0\in \Lambda_{\underline{\eta}}$. 
 We compute
$$ 
h \circ Ev_{\underline{\eta}} (g\cdot \gamma_0)=h_{g\gamma_0 \underline{\eta}}=h_{g\gamma_0 g^{-1} g\underline{\eta}}
=h_{g\underline{\eta}}=h\circ Ev_{\underline{\eta}}(g)\,,
$$
where we have used the left $\Lambda_{g \underline{\eta}}$-invariance of $h$ 
on the reflections of $g\eta$ in the third equality. 
(Recall, $g\gamma_0g^{-1}\in g\Lambda_{\underline{\eta}} g^{-1}=\Lambda_{g\underline{\eta}}$.) 
Thus, $h\circ Ev_{\underline{\eta}}:G\rightarrow G$ is invariant 
under the right action of $\Lambda_{\underline{\eta}}$. 
Since the latter group is dense in $G$, it acts ergodically on $G$ and 
$h\circ Ev_{\underline{\eta}}$ is essentially constant. This means that also $h$ is essentially constant. 
Thus, for almost every regular simplex $\underline{\xi}\in T$, 
the evaluation of $\varphi$ on any orbit point of the vertices of $\underline{\xi}$ 
under the reflection group 
$\Lambda_{\underline{\xi}}$ is equal to $h$. 
In particular, for almost every $\underline{\xi}=(\xi_0,...,\xi_n)\in T$ 
and also for almost every $\xi_0\in \mathbb{H}^n$, we have $\varphi(\xi_0)=h(\xi_0)$, 
which finishes the proof of the proposition.  
\qed
\end{proof}

We have now established that $\varphi$ is essentially equal to the isometry 
$h\in \mathrm{Isom}(\mathbb{H}^{n})$ on  
$\partial \mathbb{H}^n$. It remains to see that $h$ realizes the conjugation between $\rho$ and $i$. 
Indeed, replacing  $\varphi$ by $h$ in (\ref{equ:phiequivar}) we have 
$$ 
(h \cdot i(\gamma))( \xi) =(\rho(\gamma)\cdot h) (\xi)\,,
$$
for every $\xi \in \partial \mathbb{H}^n$ and $\gamma\in \Gamma$. 
Since all maps involved ($h,i(\gamma)$ and $\rho(\gamma)$) are isometries of $\mathbb{H}^n$ 
and two isometries induce the same map on $\partial \mathbb{H}^n$ 
if and only if they are equal it follows that 
$$ h\cdot i(\gamma) \cdot h^{-1}=\rho(\gamma)$$ 
for every $\gamma \in \Gamma$, which finishes the proof of the theorem.

\begin{acknowledgement} Michelle Bucher was supported by Swiss National Science Foundation 
project PP00P2-128309/1, Alessandra Iozzi was partially supported 
by Swiss National Science Foundation project 2000021-127016/2.  
The first and third named authors thank the Institute Mittag-Leffler in Djurholm, Sweden, 
for their warm hospitality during  the preparation of this paper.
Finally, we thank Beatrice Pozzetti for a thorough reading of the manuscript and
numerous insightful remarks.
\end{acknowledgement}



\providecommand{\bysame}{\leavevmode\hbox to3em{\hrulefill}\thinspace}
\providecommand{\MR}{\relax\ifhmode\unskip\space\fi MR }
\providecommand{\MRhref}[2]{%
  \href{http://www.ams.org/mathscinet-getitem?mr=#1}{#2}
}
\providecommand{\href}[2]{#2}

\end{document}